\documentclass[onefignum,onetabnum]{siamart171218}

\usepackage{amsmath}
\usepackage{amssymb}
\usepackage{amsfonts}
\usepackage{graphicx}
\usepackage{lipsum}
\usepackage{epstopdf}
\usepackage{algorithmic}
\ifpdf
  \DeclareGraphicsExtensions{.eps,.pdf,.png,.jpg}
\else
  \DeclareGraphicsExtensions{.eps}
\fi

\newsiamremark{remark}{Remark}
\newsiamremark{hypothesis}{Hypothesis}
\crefname{hypothesis}{Hypothesis}{Hypotheses}
\newsiamthm{claim}{Claim}

\headers{Quadratic-Exponential   Coherent Quantum Control}{I.G.Vladimirov, and I.R.Petersen}

\newcommand{\sinc}{\mathrm{sinc}}
\newcommand{\tanc}{\mathrm{tanc}}

\DeclareMathAlphabet{\bit}{OML}{cmm}{b}{it}

\def\fH{\mathfrak{H}}

\def\lprod{\mathop{\overleftarrow{\prod}}}
\def\<{\leqslant}           
\def\>{\geqslant}           

\def\d{\partial}

\def\wt{\widetilde}

\def\Re{\mathrm{Re}}   
\def\Im{\mathrm{Im}}   

\def\cH{\mathcal{H}}   
\def\mA{\mathbb{A}}    
\def\mR{\mathbb{R}}    
\def\mC{\mathbb{C}}    

\def\Tr{\mathrm{Tr}}       
\def\rT{\mathrm{T}}        
\def\rF{\mathrm{F}}        
\def\bS{\mathbf{S}}

\def\bE{\mathbf{E}}    
\def\bM{\mathbf{M}}    

\def\[[[{[\![\![}
\def\]]]{]\!]\!]}

\def\bra{{\langle}}
\def\ket{{\rangle}}

\def\Bra{\left\langle}
\def\Ket{\right\rangle}

\def\re{\mathrm{e}}        
\def\rd{\mathrm{d}}        

\def\fM{\mathfrak{M}}

\def\cL{\mathcal{L}}

\def\bA{\mathbf{A}}

\def\bJ{\mathbf{J}}

\def\x{\times}
\def\ox{\otimes}

\def\fE{\mathfrak{E}}
\def\fF{\mathfrak{F}}

\def\fP{\mathfrak{P}}

\def\sA{\mathsf{A}}
\def\sB{\mathsf{B}}
\def\sC{\mathsf{C}}

\def\sL{\mathsf{L}}

\def\cF{\mathcal{F}}

\def\cK{\mathcal{K}}

\def\cC{\mathcal{C}}

\def\sQ{\mathsf{Q}}

\def\cG{\mathcal{G}}
\def\cI{\mathcal{I}}
\def\cP{\mathcal{P}}
\def\cQ{{\mathcal Q}}

\def\cA{\mathcal{A}}
\def\cB{\mathcal{B}}
\def\cE{\mathcal{E}}

\def\bL{\mathbf{L}}

\def\mH{\mathbb{H}}
\def\mS{\mathbb{S}}

\def\eps{\epsilon}
\def\Ups{\Upsilon}
\def\ups{\upsilon}

\def\sn{|\!|\!|}

\def\diag{\mathop{\mathrm{diag}}}    
\def\fT{\mathfrak{T}}
\title{Quadratic-Exponential   Coherent Feedback Control of Linear Quantum Stochastic Systems\thanks{This work is supported by the Australian Research Council  grants  DP210101938, DP200102945.}}

\author{Igor G. Vladimirov$^{\dagger}$,
\qquad
Ian R. Petersen%
\thanks{School of Engineering, Australian National University, Canberra, ACT 2601, Australia,
  \email{igor.g.vladimirov@gmail.com, i.r.petersen@gmail.com%
  }.}}

\begin{document}

\maketitle

\begin{abstract}
This paper considers a risk-sensitive  optimal control  problem for a field-mediated interconnection of a quantum plant with a coherent (measurement-free) quantum controller. The plant and the controller are multimode open quantum harmonic oscillators governed by linear quantum stochastic differential equations,  which are coupled to each other and driven by multichannel  quantum Wiener processes modelling the external bosonic fields. The control objective is to internally stabilize the closed-loop system and minimize the infinite-horizon  asymptotic growth rate of  a quadratic-exponential functional which penalizes the plant variables and the controller output.  We obtain
first-order necessary conditions of optimality for this problem by computing the partial Frechet derivatives of the cost functional with respect to the energy and coupling matrices of the controller in frequency domain and state space. An infinitesimal equivalence between the risk-sensitive and weighted coherent quantum LQG control problems is also established. In addition to variational methods, we  employ spectral factorizations and infinite cascades of auxiliary classical systems. Their truncations are applicable to numerical optimization algorithms (such as the gradient descent) for coherent quantum risk-sensitive feedback synthesis. 
\end{abstract}



%

\section{Introduction}

Models of open quantum dynamics, which extend isolated quantum mechanical  systems towards  interaction with environment (including  other quantum or classical systems,  such as measuring devices,  and quantum fields  \cite{BP_2006,GZ_2004}) often use the Hudson-Parthasarathy quantum stochastic  calculus \cite{HP_1984,P_1992,P_2015}.  Such models employ quantum stochastic differential equations (QSDEs) for the time evolution of dynamic variables in the form of (in general,  noncommuting)  operators on a Hilbert space driven by quantum Wiener processes. The latter are
noncommutative analogues of the classical Brownian motion \cite{KS_1991} and  represent the external bosonic fields by time-varying operators on a symmetric Fock space \cite{PS_1972}. These QSDEs reflect the unitarity of the augmented  system-field evolution and   have a specific structure involving a system Hamiltonian and operators of coupling between the system and the  external fields, which specify the energetics  of the system-field interaction. The dependence of the Hamiltonian and coupling operators on the system variables and the commutation properties of the latter lead to a particular form of the resulting QSDEs and thus affect their tractability.

An important class of tractable models 
is provided by linear QSDEs for open quantum harmonic oscillators (OQHOs). Their Hamiltonian and coupling operators are, respectively,  quadratic and linear functions of the  system variables which are organized as quantum mechanical positions and momenta satisfying the canonical commutation relations (CCRs) \cite{M_1998,S_1994}. The coefficients of the linear QSDE depend on the energy and coupling matrices (parameterizing the Hamiltonian and coupling operators) in a specific fashion and are constrained by physical realizability (PR) conditions which reflect, in particular, the preservation of the CCRs over the course of time \cite{JNP_2008,SP_2012}. Due to the linearity of the governing QSDEs, the dynamic properties of  OQHOs are similar in certain respects to those of classical linear stochastic systems, including the preservation of the Gaussian nature of quantum system states \cite{KRP_2010} in the case of vacuum input fields  \cite{P_1992}. This and other useful properties (such as the fact that cascading and feedback interconnection \cite{GJ_2009,JG_2010} of OQHOs lead to augmented OQHOs) makes OQHOs efficient building  blocks  in linear quantum systems theory \cite{NY_2017,P_2017,ZD_2022}.  The latter is concerned with performance analysis and synthesis of linear quantum stochastic systems subject to given specifications, which  include  stability, robustness with respect to unmodelled dynamics and optimality in the sense of performance criteria, such as minimization of cost functionals. In particular,  similarly to classical linear quadratic Gaussian (LQG) control  and Kolmogorov-Wiener-Hopf-Kalman filtering \cite{LS_2001}, mean square cost functionals, based on second-order moments of system variables,  are used in quantum LQG control and filtering problems \cite{B_1983,BVJ_2007,EB_2005,MP_2009,MJ_2012,NJP_2009,WM_2010,ZJ_2012} for interconnections of a quantum plant with a measurement-based classical or coherent (that is, measurement-free) quantum controller or observer.

As in the classical case, the actual properties of physical systems (and quantum states)  may differ from their nominal models postulated by the mean square optimal approaches, which makes the issue of robustness particularly  important for applications, such as quantum information processing and quantum optics \cite{NC_2000,WM_2008}. In regard to parametric uncertainties, this issue is addressed, for example,  in the $\cH_\infty$, guaranteed cost LQG quantum control and robust state generation settings \cite{JNP_2008,SPJ_2007,VPJ_2019_SICON}, while robustness to quantum statistical uncertainties is one of the main features of quantum risk-sensitive control  and filtering. The latter approach employs   time-ordered exponentials \cite{J_2004,J_2005,DDJW_2006} or  the usual  operator  exponentials \cite{AB_2018,B_1996,YB_2009,VPJ_2018a} of quadratic functions of relevant quantum variables,   which leads  to quadratic-exponential functionals (QEFs) as costs to be minimized. More precisely, similarly to its classical risk-sensitive control predecessors \cite{BV_1985,J_1973,W_1981,W_1990},  the QEF 
is organized as an exponential moment of the integral of a  positive semi-definite quadratic form of the quantum system variables over a bounded time interval. The minimization of the QEF improves the upper bound \cite{VPJ_2018b}  for the worst-case mean square costs in the presence of quantum statistical uncertainty,  when the actual system-field state deviates from its nominal (for example, vacuum field state) model and this deviation does not exceed a certain threshold  in terms of the quantum relative entropy  \cite{NC_2000,OW_2010,YB_2009}. 
A similar role is played by the QEF minimization for the upper bounds on the tail probability distribution \cite{VPJ_2018a} for the quantum trajectories.
Although the above  properties resemble the links between the classical risk-sensitive criteria and minimax LQG control \cite{DJP_2000,P_2006,PJD_2000}, they are obtained in the noncommutative quantum case using quantum probabilistic considerations \cite{H_2001,H_2018,M_1995}.

Because of the noncommutativity, the computation of the QEF (in a Gaussian state of an OQHO over a bounded time interval) is substantially more complicated in comparison with the classical case. As summarized along with complete proofs   in \cite{VPJ_2021},  it involves a quantum Karhunen-Loeve expansion of the system variables, the  randomization over an auxiliary classical Gaussian process,  and a Girsanov type representation,  leading to  a different frequency-domain formula for the infinite-horizon  asymptotic growth rate of the logarithm of the QEF for invariant  Gaussian states of stable OQHOs with vacuum input  fields. This relation expresses the  QEF growth rate in terms of an integral (over the frequency) of the log-determinant and trigonometric functions composed with the matrix-valued Fourier transforms of the real and imaginary parts of the invariant quantum covariance kernel of the system variables, thus making it less tractable than the corresponding classical $\cH_\infty$-entropy integral \cite{AK_1981,MG_1990}. Nevertheless, this frequency-domain formula has already found a preliminary application to optimality conditions for measurement-based feedback control with QEF criteria \cite{VJP_2020_IFAC}. Furthermore, a method for computing the QEF rate in state space has recently been developed in \cite{VP_2022_JFI} based on a novel spectral factorization with nonrational transfer functions using infinite cascades of auxiliary classical linear time invariant (LTI) systems  and a ``system transposition'' technique for rearranging mixed products of such systems with their duals. This rearrangement  resembles the Wick ordering \cite{J_1997,W_1950} known for the annihilation and creation operators in quantum mechanics.

The present paper aims to extend the quadratic-exponential approach from performance analysis and measurement-based control design in frequency domain to coherent quantum control synthesis in state space. More precisely, we use the QEF rate as a robust performance criterion in a coherent quantum risk-sensitive     control problem for a field-mediated feedback interconnection of a quantum plant and a quantum controller. Both of them are modelled  as multimode OQHOs governed by linear QSDEs,  which are coupled to each other and driven by multichannel  quantum Wiener processes of the external bosonic fields, so that the closed-loop system is an augmented OQHO. The control objective is to internally stabilize the closed-loop system and minimize the QEF rate as a cost functional which penalizes the plant variables and the controller output forming a ``criterion'' process.  This minimization (subject to the stability constraint) is over the energy and coupling matrices which parameterize the coherent quantum controller. In order to obtain the
first-order necessary conditions of optimality for this problem, we compute the partial Frechet derivatives of the cost with respect to the controller parameters both in frequency domain and in state space. This computation combines the frequency-domain representation of the QEF rate \cite{VPJ_2021}, 
mentioned above, with variational techniques \cite{VP_2013a,VP_2013b,VJP_2020_IFAC} and
is reduced to that of a core matrix, which encodes the Frechet derivatives of the QEF rate over the state-space matrices of the closed-loop system as if the latter were independent variables.  It is through the core matrix that a particular form of the cost enters the first-order optimality conditions. Moreover, we use the core matrix for showing that the coherent quantum risk-sensitive control problem is infinitesimally  equivalent (at the level of the first-order optimality conditions) to a weighted version of the coherent quantum LQG (CQLQG) control problem \cite{NJP_2009}. In view of the key role of the core matrix,  we develop its state-space computation using spectral factorizations in terms of infinite cascades of auxiliary classical LTI systems,  building on the results of \cite{VP_2022_JFI}.  With appropriate truncations of the cascades, this variational approach  is applicable to numerical optimization algorithms (such as the gradient descent in \cite{SVP_2017}) for coherent quantum risk-sensitive feedback synthesis. The resulting iterative procedures can be initialized with an optimal CQLQG  controller.

The paper is organized as follows.
Section~\ref{sec:sys} describes the coherent quantum feedback interconnection of a plant and a controller being considered.
Section~\ref{sec:2pCCR} discusses the two-point commutation structure of the plant and controller variables and specifies a criterion process along with their statistical properties in the invariant Gaussian state.
Section~\ref{sec:QEF} uses this process in  order to formulate the coherent quantum risk-sensitive control problem with the QEF rate as a cost functional.
Section~\ref{sec:opt} establishes first-order necessary conditions of optimality for this problem in the frequency domain using the core matrix.
Section~\ref{sec:infeq} discusses the infinitesimal equivalence of the coherent quantum risk-sensitive control problem to a weighted CQLQG control problem. Section~\ref{sec:statespace} outlines a state-space computation of the core matrix using spectral factorizations and an infinite cascade of classical linear systems.
Section~\ref{sec:conc} makes concluding remarks.
Appendices~\ref{sec:matfun} and \ref{sec:mixint} provide a subsidiary material on differentiation of functions of matrices and state-space computation  of mixed moments of transfer functions.

\section{Quantum plant with a coherent quantum feedback}
\label{sec:sys}

We consider a quantum plant,  which is affected by an external quantum noise $w$ and coupled to the output field $\eta$ of a quantum controller. In turn,  the controller is driven by another  quantum noise $\omega$ and coupled to the output field $y$ of the plant. These  multichannel quantum fields are organized into column-vectors
\begin{equation}
\label{wwWyy}
  w:= (w_k)_{1\< k \< m_1},
  \qquad
  \omega := (\omega_k)_{1\< k\< m_2},
  \qquad
    W
    :=
    \begin{bmatrix}
        w\\
        \omega
    \end{bmatrix},
    \qquad
  y := (y_k)_{1\< k\< p_1},
  \qquad
  \eta :=(\eta_k)_{1\< k\< p_2}
\end{equation}
consisting of even numbers $m_1$, $m_2$, $m:= m_1+m_2$,  $p_1$, $p_2$  of time-varying self-adjoint operators (with time arguments omitted for brevity)   specified below.
The resulting field-mediated  coherent quantum feedback  interconnection is shown  in Fig.~\ref{fig:sys}.
\begin{figure}[htbp]
\centering
\unitlength=1.3mm
\linethickness{0.6pt}
\begin{picture}(50.00,15.00)
    \put(7.5,0){\dashbox(35,15)[cc]{}}
    \put(10,2.5){\framebox(10,10)[cc]{{}}}
    \put(15,8.5){\makebox(0,0)[cc]{\scriptsize{quantum}}}
    \put(15,6.5){\makebox(0,0)[cc]{\scriptsize{plant}}}
    \put(35,8.5){\makebox(0,0)[cc]{\scriptsize{quantum}}}
    \put(35,6.5){\makebox(0,0)[cc]{\scriptsize{controller}}}
    \put(30,2.5){\framebox(10,10)[cc]{{}}}
    \put(0,7.5){\vector(1,0){10}}
    \put(50,7.5){\vector(-1,0){10}}
    \put(20,10.5){\vector(1,0){10}}
    \put(30,4.5){\vector(-1,0){10}}
    \put(-2,7.5){\makebox(0,0)[rc]{\small$w $}}
    \put(25,12){\makebox(0,0)[cb]{\small$y $}}
    \put(52,7.5){\makebox(0,0)[lc]{\small$\omega $}}
    \put(25,3){\makebox(0,0)[ct]{\small$\eta $}}
\end{picture}
\caption{The fully quantum closed-loop system with a field-mediated interconnection of the quantum plant and the coherent quantum controller. It  interacts with the external input bosonic fields modelled by the quantum Wiener processes  $w $, $\omega $ driving the QSDEs (\ref{x_y}), (\ref{xi_eta}). The plant-controller interaction is mediated by the bosonic fields of the plant output $y$   and the controller output $\eta$.
}
\label{fig:sys}
\end{figure}
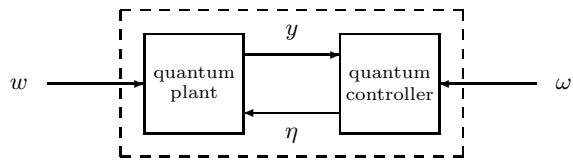
The plant and the controller are open quantum harmonic  oscillators (OQHOs) with
even numbers $n$, $\nu$  of dynamic variables $x_1, \ldots, x_n$ and $\xi_1, \ldots, \xi_\nu$ (for example, consisting of conjugate quantum mechanical  position-momentum pairs \cite{S_1994}, so that $\frac{n}{2}$, $\frac{\nu}{2}$ count the corresponding degrees of freedom).  These plant and controller variables  are time-varying self-adjoint operators on (dense domains of) a complex separable Hilbert space $\fH$, assembled into vectors  \begin{equation}
\label{xxi_X}
    x:=(x_k)_{1\< k\< n},
    \qquad
    \xi:=(\xi_k)_{1\< k\< \nu},
    \qquad
      X
  :=
  \begin{bmatrix}
    x\\
    \xi
  \end{bmatrix}
\end{equation}
and satisfying the one-point Weyl canonical commutation relations (CCRs) \cite{F_1989} in the infinitesimal Heisenberg form
\begin{equation}
\label{Theta12_XCCR}
    [x, x^{\rT}]
    =
     2i \Theta_1,
     \qquad
    [\xi, \xi^{\rT}]
    =
     2i \Theta_2,
     \qquad
     [x, \xi^{\rT}] = 0,
     \qquad
  [X,X^\rT]
  =
  2i\Theta,
  \qquad
  \Theta
  =
  \begin{bmatrix}
    \Theta_1 & 0\\
    0 & \Theta_2
  \end{bmatrix}
\end{equation}
(where all the quantum variables are considered at the same moment of time).
Here, the commutator $[\alpha, \beta]:= \alpha \beta - \beta \alpha$ of linear operators extends to the commutator matrix $[\alpha, \beta^\rT] := ([\alpha_j, \beta_k])_{1\< j \< r, 1\< k \< s}$ for vectors $\alpha:= (\alpha_j)_{1\< j \< r}$ and $\beta:= (\beta_k)_{1\< k \< s}$ of operators, $i:= \sqrt{-1}$ is the imaginary unit, and $\Theta_1 \in \mA_n$, $\Theta_2\in \mA_\nu$, $\Theta\in \mA_{n+\nu}$
are constant matrices (with $\mA_r$ the space of real antisymmetric matrices of order $r$), which are   identified with their tensor products $\Theta_k \ox \cI_\fH$, $\Theta \ox \cI_\fH$ with the identity operator $\cI_\fH$ on $\fH$. 
For what follows, the CCR matrices $\Theta_1$, $\Theta_2$ are assumed to be nonsingular, and hence,
\begin{equation}
\label{detTheta}
   \det \Theta\ne 0.
\end{equation}
The Heisenberg dynamics of the plant and controller variables in (\ref{xxi_X}) are described by Hudson-Parthasarathy \cite{HP_1984,P_1992}  linear quantum stochastic differential equations (QSDEs)
\begin{align}
\label{x_y}
    \rd x
    & =
    A x \rd t  +  B \rd w  + E \rd \eta ,
    \qquad
    \rd y
    =
    C x \rd t  +  D \rd w ,\\
\label{xi_eta}
    \rd \xi
     & =
    a\xi \rd t + b  \rd \omega  + e \rd y ,
    \qquad\quad
    \rd \eta
     =
    c \xi  \rd t + d \rd \omega,
\end{align}
with constant coefficients comprising appropriately dimensioned matrices
    $A\in \mR^{n\x n}$,
    $B\in \mR^{n\x m_1}$,
    $C\in \mR^{p_1\x n}$,
    $D\in \mR^{p_1\x m_1}$,
    $E\in \mR^{n\x p_2}$ and
    $a  \in \mR^{\nu\x \nu}$,
    $b \in \mR^{\nu\x m_2}$,
    $c \in \mR^{p_2\x \nu}$,
    $d \in \mR^{p_2\x m_2}$,
    $e \in \mR^{\nu\x p_1}$.
The plant dynamics (\ref{x_y}) are driven by the quantum Wiener process $w$ from (\ref{wwWyy}) on a symmetric Fock space \cite{P_1992} $\fF_1$ and the quantum Ito process $\eta$ of the controller output fields on the space  $\fH$. In a similar fashion,
the controller dynamics (\ref{xi_eta}) are driven by the quantum Wiener process $\omega$  from (\ref{wwWyy}) on another symmetric Fock space $\fF_2$ and the quantum Ito process $y$  of the plant output fields on the space  $\fH$. Accordingly, $\fH$ is the system-field tensor-product space given by
\begin{equation}
\label{fH}
    \fH := \fH_0 \ox \fF,
    \qquad
    \fH_0 := \fH_1\ox \fH_2,
    \qquad
    \fF := \fF_1\ox \fF_2,
\end{equation}
where $\fH_1$, $\fH_2$ are Hilbert spaces for the  initial plant and controller variables in $x(0)$, $\xi(0)$. The fact that the latter  act on different spaces explains the commutativity $[x(0),\xi(0)^\rT] = 0$ which is preserved in time as described by the third equality in (\ref{Theta12_XCCR}). The space $\fH_0$ in (\ref{fH}) is the initial space for the closed-loop system variables in (\ref{xxi_X}), so that the $n+\nu$ entries of $X(0)$ act on $\fH_0$.  The composite Fock space $\fF$ in (\ref{fH}) accommodates the augmented quantum Wiener process $W$ in (\ref{wwWyy}),
which is formed from those of the plant and the controller and has a block diagonal quantum Ito matrix  $\Omega$:
\begin{equation}
\label{WW}
    \rd W \rd W^{\rT}
    =
    \Omega \rd t,
    \qquad
    \Omega =
    \begin{bmatrix}
      \Omega_1 & 0 \\
      0 & \Omega_2
    \end{bmatrix}
    =
    I_m + iJ,
    \qquad
    \Omega_k := I_{m_k} + iJ_k,
\end{equation}
where $I_r$ is the identity matrix of order  $r$, and $\Omega_1$, $\Omega_2$ are the Ito matrices of the quantum Wiener processes $w$, $\omega$ of the plant and the controller.   The imaginary parts $J\in \mA_m$, $J_1 \in \mA_{m_1}$, $J_2\in \mA_{m_2}$, given by
\begin{equation}
\label{J}
    J
        :=
        \Im \Omega
        =
        \begin{bmatrix}
            J_1 & 0 \\
            0 & J_2
        \end{bmatrix}
        =
        I_m \ox \bJ,
        \qquad
        J_k := \Im \Omega_k = I_{m_k} \ox \bJ,
       \qquad
       \bJ
       :=
       \begin{bmatrix}
         0 & 1\\
         -1 & 0
       \end{bmatrix}
\end{equation}
(with $\bJ$ spanning the space $\mA_2$), specify 
the two-point CCRs for the quantum Wiener processes:
\begin{equation}
\label{WWcomm}
    [W(s), W(t)^\rT]
    =
    2i\min(s,t) J,
    \qquad
    s,t\>0.
\end{equation}
The future-pointing Ito increments of the quantum Wiener process $W$ (and hence, its subvectors $w$, $\omega$) commute with any adapted quantum process $\zeta$ considered at the same or an earlier moment of time:
\begin{equation}
\label{comm}
  [\rd W(t), \zeta(s)^\rT] = 0,
  \qquad
  t \> s \> 0,
\end{equation}
with the adaptedness being understood with respect to the filtration $(\fH_t)_{t\>0}$ of the system-field space $\fH$ in (\ref{fH}) given by
\begin{equation}
\label{fHt}
    \fH_t
    :=
    \fH_0 \ox \fF_t,
\end{equation}
where $\fF_t$ is the Fock subspace associated with the time interval $[0,t]$. The QSDEs (\ref{x_y}), (\ref{xi_eta}) are combined into one QSDE for the quantum process $X$ in (\ref{xxi_X}) driven by the external bosonic field $W$ in (\ref{wwWyy}):
\begin{equation}
\label{dX_cAB}
  \rd X = \cA X \rd t + \cB \rd W,
  \qquad
    \cA
  :=
  \begin{bmatrix}
    A & Ec\\
    eC & a
  \end{bmatrix},
  \qquad
  \cB
  :=
  \begin{bmatrix}
    B & Ed\\
    eD & b
  \end{bmatrix},
\end{equation}
where the closed-loop system matrices $\cA \in \mR^{(n+\nu)\x (n+\nu)}$, $\cB \in \mR^{(n+\nu)\x m}$ are computed
similarly to the classical case. 
However, since 
the plant and the controller are OQHOs, their matrices are not arbitrary and, in accordance with the physical realisability (PR) conditions \cite{JNP_2008,SP_2012} and the field-mediated coupling, are parameterized as
\begin{align}
\label{ABCE}
    A
    & =
    2\Theta_1(R_1 + M_1^{\rT}J_1 M_1 + L_1^{\rT}\wt{J}_2L_1),
    \qquad\
    B   = 2\Theta_1 M_1^{\rT},
    \qquad\
    C  =2DJ_1 M_1,
    \qquad\
    E = 2\Theta_1 L_1^{\rT},\\
\label{abce}
    a
    & =
    2\Theta_2
    (R_2 + M_2^{\rT}J_2 M_2 + L_2^{\rT}\wt{J}_1L_2),
    \qquad\ \
    b  = 2\Theta_2 M_2^{\rT},
    \qquad\ \ \,
    c = 2dJ_2 M_2,
    \qquad\ \ \,
    e  = 2\Theta_2 L_2^{\rT}.
\end{align}
Also, the feedthrough matrices $D$, $d$  in (\ref{x_y}), (\ref{xi_eta}) are formed from conjugate pairs of rows of permutation matrices of orders $m_1$, $m_2$, so that $p_1\< m_1$,  $p_2 \< m_2$, and, similarly to (\ref{J}), (\ref{WWcomm}), the matrices 
    $\wt{J}_1:= DJ_1D^{\rT} = I_{p_1/2}\ox \bJ$ and $
    \wt{J}_2:= dJ_2d^{\rT} = I_{p_2/2}\ox \bJ$
are the CCR matrices for the plant and controller output fields  $y$, $\eta$. 
Here, $R_1\in \mS_n$, $R_2\in \mS_\nu$ (with $\mS_r$ the space of real symmetric  matrices of order $r$) are energy matrices of the plant and the controller which specify their individual Hamiltonians
$\frac{1}{2} x^\rT R_1 x$,
$\frac{1}{2} \xi^\rT R_2 \xi$.   Also, $M_1\in \mR^{m_1\x n}$, $L_1\in \mR^{p_2\x n}$ are the matrices of coupling of the plant to the external input field $w$ and the controller output $\eta$,  with
${\small\begin{bmatrix}
  M_1\\ L_1
\end{bmatrix}} x$ the vector of coupling operators. Similarly, $M_2\in \mR^{m_2\x \nu }$, $L_2\in \mR^{p_1\x \nu }$ are the matrices of coupling of the controller to the external input field $\omega$ and the plant output $y$, with
${\small\begin{bmatrix}
  M_2\\ L_2
\end{bmatrix}} \xi$ the vector of coupling operators.  The QSDE in (\ref{dX_cAB}) describes an OQHO whose matrices are parameterized as
\begin{equation}
\label{cABpar}
    \cA
     =
    2\Theta(R + M^{\rT}J M),
    \qquad
    \cB   = 2\Theta M^{\rT}
\end{equation}
(with $\Theta$, $J$ the CCR matrices from (\ref{Theta12_XCCR}), (\ref{J}))
by the energy matrix $R \in \mS_{n+\nu}$,  which specifies the Hamiltonian $\frac{1}{2} X^\rT R X$ of the closed-loop system,  and  the matrix  $M \in \mR^{m \x (n+\nu)}$  of coupling between the system and the external input field $W$ in (\ref{wwWyy}), giving rise to the vector $MX$ of system-field coupling operators. These energy and coupling matrices $R$, $M$ are computed by substituting the parameterizations (\ref{ABCE}), (\ref{abce}) into (\ref{dX_cAB}), 
which leads to
\begin{equation}
\label{RMclos}
    R
    =
    \begin{bmatrix}
      R_1                                       & \frac{1}{2}(L_1^{\rT}c +C^{\rT}L_2)\\
      \frac{1}{2}(c^{\rT}L_1+L_2^{\rT}C)   & R_2
    \end{bmatrix},
    \qquad
    M
    =
    \begin{bmatrix}
      M_1           & D^{\rT}L_2 \\
      d^{\rT}L_1    & M_2
    \end{bmatrix}.
\end{equation}
The relations (\ref{RMclos}) describe the influence of the controller parameters $R_2$, $M_2$, $L_2$  on the energetics (and hence, the dynamics (\ref{dX_cAB})) of the closed-loop system, with the feedthrough matrix $d$ (which specifies the ``amount'' of the quantum noise $\omega$  in the controller output field $\eta$) being fixed.
Due to the specific structure (\ref{cABpar}) (including the symmetry of the energy matrix $R$ and the antisymmetry of the CCR matrices $\Theta$, $J$ from (\ref{Theta12_XCCR}), (\ref{J})),  the matrices $\cA$, $\cB$ satisfy 
\begin{equation}
\label{cABPR}
  \cA \Theta + \Theta \cA^\rT + \mho= 0,
  \qquad
  \mho := \cB J \cB^\rT,
\end{equation}
which is one of the PR conditions mentioned above and pertaining to the preservation of the CCRs (\ref{Theta12_XCCR}). 
The one-point CCR matrix $\Theta$ is completely specified by the commutation structure of the initial plant and controller variables and does not  depend on the energy and coupling matrices of the controller. However, the latter affect not only the closed-loop system energetics through (\ref{RMclos}), but also the two-point CCRs of the plant and controller variables.

\section{Two-point commutations, criterion process and invariant system state}
\label{sec:2pCCR}

As a solution of the linear QSDE in (\ref{dX_cAB}),  the quantum process $X$ in (\ref{xxi_X}) satisfies
$    X(t) = \re^{(t-s)\cA} X(s) + \int_s^t \re^{(t-\tau) \cA}\cB \rd W(\tau)$ for
all $    t\> s\> 0$,
which, together with the commutativity (\ref{comm}),   leads
to the following two-point CCRs (see,  for example, \cite[Eqs. (46), (47)]{VPJ_2018a}):
\begin{equation}
\label{XXcomm}
    [X(s), X(t)^\rT]
    =
    2iV(s-t),
    \qquad
    s,t\>0,
    \qquad
        V(\tau)
    :=
    \left\{
    {\begin{matrix}
    \re^{\tau \cA}\Theta &   {\rm if}\  \tau\> 0\\
    \Theta \re^{-\tau \cA^\rT}& {\rm if}\  \tau< 0
    \end{matrix}}
    \right.
    =
    (V_{jk}(\tau))_{1\< j,k\< 2},
\end{equation}
with the one-point CCRs (\ref{Theta12_XCCR}) being a particular case  since $V(0) = \Theta$. Here, the two-point commutator kernel $V$ satisfies  $V(\tau)     =
    -V(-\tau)^\rT
$
for any time difference $
    \tau \in \mR$ and is partitioned into four blocks
$    V_{11}(\tau) = -V_{11}(-\tau)^\rT \in \mR^{n\x n}$,
$
    V_{12}(\tau) = -V_{21}(-\tau)^\rT \in \mR^{n\x \nu}$, $V_{22}(\tau) = -V_{22}(-\tau)^\rT \in \mR^{\nu\x \nu}$.
The first block-column of $V$   is computed as
\begin{equation}
\label{L11dot_L21dot}
    V_{\bullet 1}(\tau)
    :=
    (V_{j1}(\tau))_{1\< j\< 2}
    =
    \re^{\tau \cA }
    \begin{bmatrix}
        \Theta_1\\
        0
    \end{bmatrix}
    =
    (\re^{\tau \cA })_{\bullet 1}\Theta_1 ,
    \qquad
    \tau \> 0,
\end{equation}
where use is made of the corresponding block-columns of $\re^{\tau \cA}$ and the block diagonal matrix $\Theta$ from (\ref{Theta12_XCCR}). In combination with the structure of the matrix $\cA$ in (\ref{dX_cAB}), the relation  (\ref{L11dot_L21dot}) leads to the one-sided time derivative $\dot{V}_{21}(0+) =    eC \Theta_1$ which depends on $e$. Furthermore, the two-point CCR kernel $V_{11}$ of the plant variables in (\ref{XXcomm}) satisfies $\ddot{V}_{11}(0+) = (\cA^2)_{11}\Theta_1 = (A^2 + EceC)\Theta_1$ and is also affected by the controller matrices.

The dependence of the two-point commutation structure of the plant and controller variables on the controller parameters plays a role in the risk-sensitive optimal control problem for the closed-loop system formulated in Section~\ref{sec:QEF} in terms of a
``criterion'' process $Z$,  which consists of $r$ time-varying self-adjoint quantum variables defined by
 \begin{equation}
\label{Z_U_cC}
  Z
  :=
  N X + K U = \cC X,
  \qquad
  U := c\xi,
  \qquad
    \cC
  :=
  \begin{bmatrix}
    N & K c
  \end{bmatrix}.
\end{equation}
Here,
$N \in \mR^{r \x n}$, $K \in \mR^{r \x p_2}$ are given weighting matrices, which specify control design preferences in regard to relative importance of the system variables and are free from PR constraints. Also, $U$ 
 is an auxiliary quantum process
which is the drift of the controller output QSDE in (\ref{xi_eta}), corresponding to (though different from) the classical actuator signal.  Accordingly, the matrix $\cC \in \mR^{r\x (n+\nu)}$ in (\ref{Z_U_cC}) depends affinely on the controller matrix $c$ from (\ref{abce}).
In view of (\ref{XXcomm}),   the criterion process $Z$ has the following two-point CCRs:
\begin{equation}
\label{ZZcomm}
    [Z(s), Z(t)^\rT]
    =
    2i\Lambda(s-t),
    \qquad
    s,t\>0,
    \qquad
  \Lambda(\tau):= \cC V(\tau) \cC^\rT.
\end{equation}
The two-point CCR function $\Lambda$ satisfies $\Lambda(\tau)= -\Lambda(-\tau)^\rT$ for any
$
  \tau \in \mR$ (similarly to $V$) and, for any given time horizon $T>0$,  specifies a compact skew self-adjoint operator $\cL_T$  on the Hilbert space $L^2([0,T],\mC^r)$ of square integrable $\mC^r$-valued functions on the time interval $[0,T]$ as
\begin{equation}
\label{cL}
    \cL_T(f)(s)
    :=
    \int_0^T
    \Lambda(s-t) f(t)
    \rd t,
    \qquad
    0\< s \< T,
\end{equation}
which will be used in Section~\ref{sec:QEF}.
Note that the commutation structure (\ref{XXcomm})   of the plant and controller variables (and also (\ref{ZZcomm}) for the process $Z$) does not depend on the system-field quantum state.

 In what follows, we will be concerned
with the case of vacuum input fields, so that the underlying density operator  $\rho$  on the system-field space $\fH$ in (\ref{fH}) is given by
$\rho:= \rho_0\ox \ups$,
where $\rho_0$ is the initial  quantum state of the closed-loop system on $\fH_0$, and $\ups$ is the vacuum field state \cite{P_1992} on the Fock space $\fF$.   Then, by the results of \cite{VPJ_2018a} applied to the closed-loop system with any  \emph{stabilizing} controller (which makes the matrix $\cA$ in (\ref{dX_cAB}) Hurwitz), the plant and controller  variables have a unique  invariant multipoint zero-mean Gaussian quantum state. Assuming that the system is initialized at the invariant state (and hence, retains it over the course of time), $X$ is a zero-mean stationary Gaussian quantum process.  These properties are inherited from $X$ by the process $Z$ in (\ref{Z_U_cC}), and its two-point quantum covariance function takes the form
\begin{equation}
\label{EZZ}
\bE(Z(s)Z(t)^\rT)
  =
  P(s-t) + i\Lambda(s-t),
  \qquad
  s,t\> 0,
\end{equation}
where $\bE \zeta := \Tr(\rho \zeta)$ is the quantum expectation over the density operator $\rho$. While the imaginary part $\Lambda$ is given by (\ref{ZZcomm}) irrespective of the quantum state, the real part $P$  is associated with the invariant Gaussian state:
\begin{equation}
\label{P}
    P(\tau)
    =
    \left\{
    {\begin{matrix}
    \cC\re^{\tau \cA}\cP \cC^\rT& {\rm if}\  \tau\> 0\\
    \cC \cP\re^{-\tau \cA^\rT}\cC^\rT & {\rm if}\  \tau < 0
    \end{matrix}}
    \right.
    =
    P(-\tau)^\rT,
    \qquad
    \tau \in \mR.
\end{equation}
Here, the matrix $\cP \in \mS_{n+\nu}^+$ (with $\mS_r^+$ the set of real positive semi-definite symmetric matrices of order $r$) describes the invariant one-point covariances of the plant and controller variables and coincides with the controllability Gramian \cite{KS_1972}  of the pair $(\cA, \cB)$,  which is the  unique  solution
\begin{equation}
\label{cP}
    \cP
    :=
    \Re \bE (X X^\rT)
    =
    \int_0^{+\infty}
    \re^{t\cA }
    \cB\cB^\rT
    \re^{t\cA^\rT }
    \rd t
    =:
    \bL_\cA(\cB\cB^\rT)
\end{equation}
of the algebraic Lyapunov equation (ALE) $\cA \cP + \cP \cA^\rT + \cB\cB^\rT=0$ and
depends on $\cB\cB^\rT$ through the linear  operator $\bL_\cA$. 
  Moreover, $P+i\Lambda$ in (\ref{EZZ}) is a positive semi-definite Hermitian kernel,  whose Fourier transform
\begin{align}
\label{QSD}
    \int_\mR
    \re^{-i\lambda t }
    (P(t)+i\Lambda(t))
  \rd t
    =
  \Phi(\lambda)+i\Psi(\lambda)
  =
  F(i\lambda) \Omega F(i\lambda)^*
  \succcurlyeq
  0,
  \qquad
  \lambda \in \mR,
\end{align}
takes values in the set $\mH_r^+$ of positive semi-definite matrices in the space $\mH_r = \mS_r + i\mA_r$ of complex Hermitian matrices of order $r$ and plays the role of a ``quantum spectral density'' of the stationary Gaussian quantum process $Z$. Here, $(\cdot)^*:= {{\overline{(\cdot)}}}^\rT$ is the complex conjugate transpose, and use is made of the quantum Ito matrix $\Omega$ from (\ref{WW}) along with
\begin{equation}
\label{Phi0_Psi0}
    \Phi(\lambda)
     :=
    \int_\mR \re^{-i\lambda t }
    P(t)
    \rd t
    =
    F(i\lambda) F(i\lambda)^*,
    \qquad
    \Psi(\lambda)
     :=
    \int_\mR \re^{-i\lambda t }
    \Lambda(t)
    \rd t
    =
    F(i\lambda) J F(i\lambda)^*,
\end{equation}
where $J$ is the matrix from (\ref{J}). 
The right-hand sides of (\ref{QSD}), (\ref{Phi0_Psi0}) are factorized in terms of a strictly proper rational transfer function $F$ with the state-space realization triple $(\cA,\cB,\cC)$  for the closed-loop system (\ref{dX_cAB}), (\ref{Z_U_cC})  from the incremented quantum Wiener process $W$ to the process $Z$:
\begin{equation}
\label{F0_G}
    F(s)
    :=
    \cC
    G(s)
    \cB,
    \qquad
    G(s):= (sI_{n+\nu} - \cA)^{-1},
    \qquad
    s \in \mC,
\end{equation}
where $G$ coincides with the negative of the resolvent for the matrix $\cA$. 
Note that $\Phi(\lambda) \in\mH_r^+$ in (\ref{Phi0_Psi0}), while  $\Psi(\lambda)$   is skew Hermitian at any frequency $\lambda \in \mR$.

\section{Quadratic-exponential performance criterion}
\label{sec:QEF}

The risk-sensitive control setting considered below starts from describing the coherent quantum feedback performance over a finite time horizon $T>0$ in terms of the quadratic-exponential functional (QEF)
\begin{equation}
\label{Xi}
    \Xi_{\theta,T}
    :=
    \bE \re^{\frac{\theta}{2} Q_T}.
\end{equation}
Here, $\theta>0$ is a risk sensitivity parameter specifying the exponential penalty on the following positive semi-definite self-adjoint quantum variable on the subspace $\fH_T$  in (\ref{fHt}):
\begin{equation}
\label{QT}
    Q_T
    :=
    \int_0^T
    Z(t)^\rT Z(t)
    \rd t
    =
    \int_0^T
    X(t)^\rT \cC^\rT \cC X(t)
    \rd t.
\end{equation}
The latter depends in a quadratic fashion  on  the history of the plant and controller variables in (\ref{xxi_X}) over the time interval $[0,T]$ through the criterion process $Z$ from (\ref{Z_U_cC}). Accordingly, the usual mean square cost functionals are a limiting case of (\ref{Xi}) as
\begin{equation}
\label{Xilim}
    \lim_{\theta \to 0+}
    \frac{\Xi_{\theta,T}-1}{\theta}
    =
    \lim_{\theta \to 0+}
    \frac{\ln \Xi_{\theta,T}}{\theta}
    =
    \frac{1}{2}
    \bE Q_T.
\end{equation}
In particular, (\ref{QT}) acquires the form $Q_T= \int_0^T (X(t)^\rT \Pi_1 X(t) + U(t)^\rT \Pi_2 U(t))\rd t$ by letting
$N := {\small\begin{bmatrix}
  \sqrt{\Pi_1}\\
  0
\end{bmatrix}}$ and $K := {\small\begin{bmatrix}
  0 \\
  \sqrt{\Pi_2}
\end{bmatrix}}$, where  $\Pi_1 \in \mS_n^+$, $\Pi_2\in \mS_{p_2}^+$, in which case (\ref{Xi}) imposes a quadratic-exponential penalty on the plant variables and the controller output variables in the conventional form of classical risk-sensitive cost functionals \cite{J_1973,W_1981}. 
Alternatively,  a quantum risk-sensitive filtering problem is  obtained by letting $E=0$ (which decouples the plant dynamics (\ref{x_y}) from the controller), $N: = \sqrt{\Pi}$ and $K:= -N$, where $\Pi\in \mS_n^+$. In this case, $r=n=p_2$, and  (\ref{QT}) takes the form   $Q_T=\int_0^T (X(t)-U(t))^\rT \Pi (X(t)-U(t))\rd t$. Accordingly,  the controller becomes a coherent  quantum observer \cite{MJ_2012,VP_2013b}, with its output drift $U$ in (\ref{Z_U_cC})  playing the role of an estimator for the plant variables, so that (\ref{Xi}) penalizes the resulting ``estimation error'' process $X-U$ (see also \cite{YB_2009} for measurement-based risk-sensitive filtering formulations).

Returning to the general coherent quantum control setting with arbitrary weighting matrices $N$, $K$ in (\ref{Z_U_cC}), suppose the coherent quantum controller is stabilizing, the input fields are in the vacuum state, and the closed-loop system is in the invariant Gaussian state described in Section~\ref{sec:2pCCR}.  
Then, by \cite[Theorem~8.1]{VPJ_2021} applied to the stationary Gaussian quantum process $Z$ in (\ref{Z_U_cC}),  under  the condition that it is ``completely noncommutative'' in the sense that the operator $\cL_T$ in (\ref{cL}) has no zero eigenvalues,
\begin{equation}
\label{nozer}
  \ker \cL_T =\{0\}
  \quad
  {\rm for\ all\ sufficiently\ large}\ T>0,
\end{equation}
the infinite-horizon  asymptotic behaviour of the QEF $\Xi_{\theta,T}$ in (\ref{Xi}) is described by the growth rate
\begin{equation}
\label{Ups}
    \Ups_\theta
    :=
\lim_{T\to +\infty}
    \Big(
        \frac{1}{T}
        \ln \Xi_{\theta,T}
    \Big)
     =
    -
    \frac{1}{4\pi}
    \int_{\mR}
    \ln\det
    D_\theta(\lambda)
    \rd \lambda.
\end{equation}
Here,
\begin{equation}
\label{D}
    D_\theta(\lambda)
     :=
    c_\theta(\lambda)-
        \theta
        \Phi(\lambda)
        \sinc
        (\theta \Psi(\lambda))
        =
    c_\theta(\lambda)
-
        \Phi(\lambda)
        \Psi(\lambda)^{-1}
        s_\theta(\lambda)
\end{equation}
(with $\sinc z := \frac{\sin z}{z}$ an entire function extended  to $\sinc 0 := 1$ by continuity)
is a $\mC^{r\x r}$-valued function on $\mR$ defined using
\begin{equation}
\label{cs}
    c_\theta(\lambda)
    :=
      \cos(
        \theta \Psi(\lambda)
    ),
    \qquad
    s_\theta(\lambda)
    :=
      \sin(
        \theta \Psi(\lambda)
    )
\end{equation}
and the Fourier transforms (\ref{Phi0_Psi0}) of the covariance and commutator kernels $P$, $\Lambda$ from (\ref{P}), (\ref{ZZcomm}). The limit in (\ref{Ups}) holds for all sufficiently small values of the risk sensitivity parameter $\theta>0$ in (\ref{Xi}) in the sense that
\begin{equation}
\label{spec1}
    \theta
    \sup_{\lambda \in \mR}
    \lambda_{\max}
    (
        \Phi(\lambda)
        \tanc
        (\theta \Psi(\lambda))
    )
    < 1,
\end{equation}
where $\lambda_{\max}(\cdot)$ is the largest eigenvalue.

As a function of $\theta$ (subject to (\ref{spec1})), the QEF rate (\ref{Ups}) can be used for quantifying the large deviations of quantum trajectories \cite{VPJ_2018a} of the closed-loop system (in the form of upper bounds on the tail distributions of $Q_T$ in (\ref{QT})) and its robustness to statistical uncertainties with a quantum relative entropy \cite{OW_2010} description    (see \cite[Section~IV]{VPJ_2018b}
 and references therein including \cite{YB_2009}). These quantum robustness bounds are similar to those in minimax LQG control of classical stochastic systems \cite{DJP_2000,P_2006,PJD_2000} and depend on the QEF rate $\Ups_\theta$ in a monotonic fashion, so that smaller values of $\Ups_\theta$ yield stronger bounds. Due to this monotonicity, the minimization of $\Ups_\theta$ over the controller parameters is beneficial for robust performance of the closed-loop system. 
 The resulting infinite-horizon risk-sensitive control problem  is formulated for a given $\theta>0$   as the minimization
\begin{equation}
\label{opt}
  \Ups_\theta \longrightarrow\inf
\end{equation}
of the QEF rate (\ref{Ups}) over the energy and coupling matrices $R_2$, $M_2$, $L_2$ of the coherent quantum controller in (\ref{xi_eta}), (\ref{abce}) subject to the internal stability (so as to make $\cA$ in (\ref{dX_cAB}) Hurwitz) and the spectral constraint (\ref{spec1}).

 In comparison with classical risk-sensitive control \cite{BV_1985,J_1973,W_1981}, the controller parameters influence the cost functional (\ref{Ups}) not only through the statistical properties of the plant and controller variables, captured by the function $\Phi$ in (\ref{Phi0_Psi0}), but also through their two-point commutation  structure described by $\Psi$, which enters (\ref{Ups}) in view of (\ref{D}).

The development of conditions of optimality for  the problem (\ref{opt}) in Section~\ref{sec:opt} employs a spectral density representation of the QEF rate provided by Lemma~\ref{lem:UpsDel} below. To this end,
in view of (\ref{cs}), the matrix  (\ref{D}) is expressed in terms of the matrix exponentials $
    c_\theta(\lambda) \pm i s_\theta(\lambda)  = \re^{\pm i\theta \Psi(\lambda)} \succ 0
$ as
\begin{equation}
\label{Dexp}
    D_\theta
     =
    (
        I_r -\theta (\Phi - i\Psi)\phi(2i\theta \Psi)
    )
    \re^{-i\theta \Psi}
\end{equation}
(the frequency argument $\lambda \in \mR$ is omitted for brevity)
similarly to \cite[Eq. (51)]{VP_2022_JFI}. Here,
$\phi$  is an entire function,  given by
\begin{equation}
\label{phi}
  \phi(u)
  :=
  \frac{\re^{u}-1}{u}
  =
  \sum_{k=0}^{+\infty}
  \phi_k u^k,
  \qquad
  \phi_k:= \frac{1}{(k+1)!},
  \qquad
  u  \in \mC
\end{equation}
(with $\phi(0):= 1$ by continuity) and
taking positive values on the real line: $\phi(\mR)\subset (0, +\infty)$.
In (\ref{Dexp}),   the function $\phi$ maps the matrix $2i\theta \Psi(\lambda) \in \mH_r$ to  a positive definite Hermitian matrix:
\begin{equation}
\label{phipos}
    \phi(2i\theta \Psi(\lambda)) \succ 0 ,
    \qquad
    \lambda \in \mR.
\end{equation}
From (\ref{Phi0_Psi0}), with the transfer matrices $F(s)$,  $G(s)$ in (\ref{F0_G}) evaluated at $s:= i\lambda$ with $\lambda \in \mR$ as before,  it follows that
\begin{equation}
\label{PhiPsiEE}
    \Phi-i\Psi = F\Omega^\rT F^* = FS^2F^*. 
\end{equation}
Here, the first equality is similar to the second equality in  (\ref{QSD}), and  use is made of  a square root
\begin{equation}
\label{Oroot}
  S:=
  \sqrt{\Omega^\rT}
  =
  \frac{1}{\sqrt{2}} \Omega^\rT\in \mH_m^+
\end{equation}
of the matrix $    \Omega^\rT = I_m - iJ = \overline{\Omega}$ due to the structure of the quantum Ito matrix $\Omega$ in (\ref{WW}) and its imaginary part $J$ in (\ref{J})
(indeed, since $J^2 = -I_m$, then  $(\Omega^\rT)^2 = I_m - J^2 - 2iJ =  2\Omega^\rT$, which also implies that $\frac{1}{2}\Omega^\rT$ is idempotent). 
In view of (\ref{PhiPsiEE}),  the matrix
$    (\Phi - i\Psi)\phi(2i\theta \Psi)
     =
    F S^2 F^*\phi(2i\theta \Psi)
$ is, up to zero eigenvalues,    isospectral  to
\begin{equation}
\label{Sig}
    \Sigma_\theta
    :=
    S F^*
     \phi(2i\theta \Psi)
     F S
\end{equation}
obtained by swapping the factors $F S$ and $S F^*\phi(2i\theta \Psi)$. The resulting
function $\Sigma_\theta: \mR\to \mH_m^+$ has all the properties of the spectral density of a classical $\mC^m$-valued random process.
Indeed, the Hermitian property and positive semi-definiteness of the matrix $\Sigma_\theta(\lambda)$ at any frequency $\lambda$ follow from (\ref{phipos}) and the Hermitian property of $S$ in (\ref{Oroot}). 
The significance of $\Sigma_\theta$ for computing the QEF rate (\ref{Ups}) is based on  the following lemma, similar to \cite[Lemma 1]{VP_2022_JFI}. For its formulation, we associate  with $\Sigma_\theta$ a function $\Delta_\theta: \mR \to \mH_m$ given by
\begin{equation}
\label{Del}
    \Delta_\theta
     :=
     I_m - \theta \Sigma_\theta.
\end{equation}
Note that the fulfillment of the condition (\ref{spec1}) makes the matrix $\Delta_\theta(\lambda)$ uniformly positive definite over $\lambda \in \mR$ in the sense that its smallest eigenvalue is separated from zero:
\begin{equation}
\label{unipos}
    \inf_{\lambda\in \mR}
    \lambda_{\min}
    (\Delta_\theta
    (\lambda)) >0.
\end{equation}
Therefore, $\Delta_\theta$ is a spectral density function which, due to the decay of $\Sigma_\theta$ in (\ref{Sig}) at infinity,  satisfies $\lim_{\lambda \to \infty} \Delta_\theta(\lambda) = I_m$.

\begin{lemma}
\label{lem:UpsDel}
Suppose the matrix $\cA$ of the closed-loop system in (\ref{dX_cAB}), (\ref{Z_U_cC}),  is Hurwitz,   and the conditions (\ref{nozer}), (\ref{spec1}) are satisfied. Then the QEF rate (\ref{Ups}) is represented in terms of (\ref{Del})  as
\begin{equation}
\label{UpsDel}
    \Ups_\theta
     =
    -
    \frac{1}{4\pi}
    \int_{\mR}
    \ln\det
    \Delta_\theta(\lambda)
    \rd \lambda.
\end{equation}
\end{lemma}
\begin{proof}
From the isospectrality,  up to zero eigenvalues, of $(\Phi - i\Psi)\phi(2i\theta \Psi)$ to $\Sigma_\theta$ in (\ref{Sig}), it follows that $I_r - \theta (\Phi - i\Psi)\phi(2i\theta \Psi)$ is isospectral to $\Delta_\theta$ in (\ref{Del}) up to unit eigenvalues, and hence,
\begin{equation}
\label{lndet0}
    \ln\det
    (
        I_m -
        \theta
        (\Phi - i\Psi)\phi(2i\theta \Psi)
    )
    =
    \ln\det \Delta_\theta.
\end{equation}
A combination of (\ref{lndet0}) with the representation (\ref{Dexp}) of the function $D_\theta$ in (\ref{D}) and the identity $\det \re^\mu = \re^{\Tr \mu}$ for square matrices $\mu$ leads to
\begin{equation}
\label{lndetD}
        \ln\det
    D_\theta(\lambda)
     =
    \ln\det
    \Delta_\theta(\lambda)
    -i\theta \Tr \Psi(\lambda),
    \qquad
    \lambda \in \mR.
\end{equation}
Since $\Tr \Psi$ is the Fourier transform of $\Tr \Lambda$ by (\ref{Phi0_Psi0}), application of the inverse Fourier transform yields
$
    \int_\mR
    \Tr \Psi(\lambda)
    \rd \lambda
    =
    2\pi \Tr \Lambda(0) = 0
$ in view of the one-point CCR matrix $\Lambda(0) = \cC V(0)\cC^\rT = \cC \Theta \cC^\rT$ in (\ref{ZZcomm}) being traceless (as any antisymmetric matrix). The representation (\ref{UpsDel}) is now obtained by integrating both sides of (\ref{lndetD}) over $\lambda$ and recalling (\ref{Ups}).
\end{proof}

Note that the transfer function $G$ (and hence, $\Psi$ in (\ref{Phi0_Psi0})) is strictly proper, and $\phi(0)=1$,  whereby $\|\Sigma_\theta(\lambda)\| = O(1/\lambda^2)$ (for any matrix norm $\|\cdot\|$) as $\lambda \to \infty$, thus making $\Sigma_\theta$ in (\ref{Sig}) absolutely integrable:
$
  \int_\mR
  \|\Sigma_\theta(\lambda)\|
  \rd \lambda
  <
  +\infty
$. This integrability  and the uniform positive definiteness (\ref{unipos}) under the condition (\ref{spec1})  secure the convergence of the integral in (\ref{UpsDel}). Moreover, due to these properties, the QEF rate $\Ups_\theta$ is a  Frechet differentiable function of the closed-loop system matrices $\cA$, $\cB$, $\cC$. In comparison with the original formula (\ref{Ups}), the advantage of the representation (\ref{UpsDel}) is that its integrand uses Hermitian matrices, which will simplify the computation of the Frechet derivatives.

\section{First-order necessary conditions of optimality}
\label{sec:opt}

Since the closed-loop system matrix $\cA$ in (\ref{dX_cAB}) is Hurwitz for any stabilizing coherent quantum  controller (\ref{xi_eta}), then, in view of the parameterization (\ref{abce}),   the cost functional $\Ups_\theta$, defined by  (\ref{Ups}) and endowed with an equivalent representation (\ref{UpsDel}) in Lemma~\ref{lem:UpsDel},  is an infinitely differentiable  composite function
\begin{equation}
\label{comp}
    \fE:=
    \mS_\nu\x \mR^{m_2\x \nu}\x \mR^{p_1\x \nu}
    \ni
    \cE
    :=
    (R_2, M_2, L_2)
    \mapsto
    (a,b,c,e)
    \mapsto
    (\cA,\cB,\cC)
    \mapsto
    \Ups_\theta
\end{equation}
of the triple $\cE$ of the energy and coupling matrices of such a controller for any given $\theta>0$   subject to (\ref{spec1}).  Those triples $\cE$, which specify a stabilizing coherent quantum controller (\ref{xi_eta}), (\ref{abce}) satisfying (\ref{nozer}), (\ref{spec1}), form an open subset of the space $\fE$. Such controllers will be referred to as \emph{admissible} controllers.

Similarly to the variational approach to quantum control and filtering problems with mean square performance  criteria  \cite{VP_2013a,VP_2013b}, the first-order necessary conditions of optimality for the risk-sensitive  control problem (\ref{opt}) in the class of admissible controllers can be obtained by equating to zero the partial Frechet derivatives of the QEF rate $\Ups_\theta$  with respect to the energy and coupling matrices $R_2$, $M_2$, $L_2$ of the controller.
A frequency-domain representation of these derivatives is provided by Theorem~\ref{th:ders} below. For its formulation, the affine dependence of the closed-loop system matrices $\cA$, $\cB$, $\cC$ in (\ref{dX_cAB}), (\ref{Z_U_cC}) on the controller matrices $a$, $b$, $c$, $e$ in the QSDEs (\ref{xi_eta}), that is,  the second intermediate map in (\ref{comp}),  is represented (similarly to \cite[Eqs. (22), (23)]{VP_2013a}) as
\begin{equation}
\label{Gamgam}
    \Gamma
    :=
    \begin{bmatrix}
      \cA & \cB \\
      \cC & 0
    \end{bmatrix}
    =
    \Gamma_0
    +
    \Gamma_1
    \gamma
    \Gamma_2,
    \qquad
    \gamma
    :=
    \begin{bmatrix}
      a & b & e \\
      c & 0 & 0
    \end{bmatrix},
\end{equation}
where $\Gamma_0 \in \mR^{(n+\nu+ r) \x (n+\nu+m)}$, $\Gamma_1 \in \mR^{(n+\nu+ r) \x (\nu+p_2)}$,   $\Gamma_2 \in \mR^{(\nu+m_2+p_1) \x (n+\nu+m)}$ are auxiliary matrices given by
\begin{equation}
\label{GGG}
  \Gamma_0
  :=
    \begin{bmatrix}
      A & 0 & B & Ed\\
      0 & 0_\nu & 0 & 0\\
      S & 0 & 0 & 0
    \end{bmatrix},
    \qquad
  \Gamma_1
  :=
    \begin{bmatrix}
      0 & E\\
      I_\nu & 0\\
      0 & K
    \end{bmatrix},
    \qquad
  \Gamma_2
  :=
    \begin{bmatrix}
      0 & I_\nu & 0 & 0\\
      0 & 0   & 0 & I_{m_2}\\
      C & 0 & D & 0
    \end{bmatrix},
\end{equation}
with $0_s$ the $(s\x s)$-matrix of zeros (the dimensions of the other blocks are recovered from their surroundings).
In (\ref{Gamgam}), the closed-loop system and controller matrices  are assembled into the matrices $\Gamma \in \fT_{n+\nu,m,r}$ and $\gamma\in \fT_{\nu,m_2+p_1,p_2}$ in the corresponding subspaces
\begin{equation}
\label{fT}
    \fT_{s,\mu,\rho}
    :=
    \left\{
    \begin{bmatrix}
      \varphi & \sigma\\
      \tau & 0
    \end{bmatrix}:\
    \varphi \in \mR^{s\x s},\
    \sigma \in \mR^{s\x \mu},\
    \tau \in \mR^{\rho \x s}
    \right\}
\end{equation}
of the Hilbert space $\mR^{(s+\rho)\x (s+\mu)}$,
where $s$, $\mu$, $\rho$ specify the block dimensions. The operator of projection from $\mR^{(s+\rho)\x (s+\mu)}$ onto the subspace $\fT_{s,\mu,\rho}$ will be denoted as  $\daleth_{s,\mu,\rho}$. Its action
\begin{equation}
\label{daleth}
  \daleth_{s,\mu,\rho}
  \left(
    \begin{bmatrix}
      \varphi & \sigma\\
      \tau & \psi
    \end{bmatrix}
  \right)
  :=
    \begin{bmatrix}
      \varphi & \sigma\\
      \tau & 0
    \end{bmatrix}
\end{equation}
consists in padding the bottom-right $(\rho\x \mu)$-block $\psi$ of the matrix with zeros, in accordance with the sparsity structure in (\ref{fT}). The operator $\daleth_{s,\mu,\rho}$ and the target subspace $\fT_{s,\mu,\rho}$ extend to the case of complex matrices in a natural fashion.
In view of the affine dependence of the matrix $\Gamma$ on $\gamma$ in (\ref{Gamgam}), the corresponding Frechet derivative  can be  represented as
\begin{equation}
\label{dGamdgam}
  \d_{\gamma}\Gamma
  =
  \[[[
    \Gamma_1,
    \Gamma_2
  \]]].
\end{equation}
Here, $\[[[\sigma, \tau\]]]$ denotes a ``sandwich'' operator (playing a role in algebraic Sylvester equations \cite{GLAM_1992,SIG_1998}), which is specified by real or complex  matrices $\sigma$, $\tau$ and acts on an appropriately dimensioned  matrix $\alpha$ as
\begin{equation}
\label{sand}
    \[[[\sigma, \tau\]]](\alpha)
    :=
    \sigma
    \alpha
    \tau.
\end{equation}
In a similar fashion, $\[[[\sigma_1, \tau_1 \mid \ldots \mid \sigma_s, \tau_s\]]]:= \sum_{k=1}^s \[[[\sigma_k, \tau_k\]]]$ defines a ``multisandwich'' operator.
The operator $\d_{\gamma}\Gamma$ in (\ref{dGamdgam}) provides a concise representation (in fact, makes advantage of the sparsity) of the
Jacobian matrix  of the affine map $    (a,b,c,e)
    \mapsto
    (\cA,\cB,\cC)
$ in (\ref{dX_cAB}), (\ref{Z_U_cC}):
$  \d_{a,b,c,e}(\cA, \cB, \cC)
  =
  {\scriptsize\begin{bmatrix}
    \d_a\cA & 0 & \d_c\cA & \d_e\cA \\
    0 & \d_b\cB & 0 & \d_e\cB \\
    0 & 0 & \d_c\cC & 0
  \end{bmatrix}}
$,
 which is computed  as if
the controller matrices $a$, $b$, $c$, $e$ were  independent variables.
This Jacobian matrix consists of linear operators,  
and its nontrivial entries are the partial Frechet derivatives
\begin{align*}
    \d_a\cA
    & =
    \[[[
        \begin{bmatrix}
          0 \\
          I_\nu
        \end{bmatrix},
        \begin{bmatrix}
          0 & I_\nu
        \end{bmatrix}
    \]]],
    \qquad\,\, \,
    \d_c\cA =
    \[[[
        \begin{bmatrix}
          E \\
          0
        \end{bmatrix},
        \begin{bmatrix}
          0 & I_\nu
        \end{bmatrix}
    \]]],
    \qquad
    \d_e\cA =
    \[[[
        \begin{bmatrix}
          0 \\
          I_\nu
        \end{bmatrix},
        \begin{bmatrix}
          C & 0
        \end{bmatrix}
    \]]],\\
    \d_b\cB
    & =
    \[[[
        \begin{bmatrix}
          0 \\
          I_\nu
        \end{bmatrix},
        \begin{bmatrix}
          0 & I_{m_2}
        \end{bmatrix}
    \]]],
    \qquad
    \d_e\cB =
    \[[[
        \begin{bmatrix}
          0 \\
          I_\nu
        \end{bmatrix},
        \begin{bmatrix}
          D & 0
        \end{bmatrix}
    \]]],
    \qquad 
    \d_c\cC
    =
    \[[[
        K,
        \begin{bmatrix}
          0 & I_\nu
        \end{bmatrix}
    \]]],
\end{align*}
organized as sandwich operators. Theorem~\ref{th:ders} will also employ the following functions $\phi_\theta, \psi_\theta: \mR\to \mC^{r\x m}$ and $\varpi_\theta: \mR\to \mH_r$: 
\begin{align}
\label{phitheta_psitheta}
    \phi_\theta(\lambda)
    & :=
    \phi(2i\theta\Psi(\lambda)) F(i\lambda) S\Delta_\theta(\lambda)^{-1} S,
    \qquad
    \psi_\theta(\lambda)
    :=
    \varpi_\theta(\lambda)
          F(i\lambda)J,\\
\label{varpitheta}
    \varpi_\theta(\lambda)
    & :=
          \left(
          \phi
          \left(
          \begin{bmatrix}
            2i\theta\Psi(\lambda) & 0\\
            F(i\lambda) S\Delta_\theta(\lambda)^{-1} SF(i\lambda)^*& 2i\theta\Psi(\lambda)
          \end{bmatrix}
          \right)
          \right)_{21},
          \qquad
          \lambda \in \mR,
\end{align}
defined in terms of (\ref{J}), (\ref{Phi0_Psi0}), (\ref{F0_G}), (\ref{phi}), (\ref{phipos}),  (\ref{Oroot}), (\ref{Del}), with $(\cdot)_{21}$ denoting  the bottom left $( r\x  r)$-block  of the appropriately partitioned  $(2r\x 2r)$-matrix.  
Also, we will use a matrix
\begin{equation}
\label{chitheta}
    \chi_\theta
    :=
  \frac{1}{2\pi}
  \int_{\mR}
  \Re
  \fP
  \left(
    \begin{bmatrix}
      G^* \cC^\rT \\
      I_r
    \end{bmatrix}
    (\phi_\theta + 2i\theta \psi_\theta)
    \begin{bmatrix}
      \cB^\rT G^* &I_m
    \end{bmatrix}
      \right)
      \rd \lambda
      \in
      \fT_{n+\nu,m,r},
\end{equation}
where,  in accordance with (\ref{fT}),  (\ref{daleth}), the operator
\begin{equation}
\label{fP}
    \fP:= \daleth_{n+\nu,m,r}
\end{equation}
is the projection onto the subspace $\fT_{n+\nu,m,r}$ of matrices whose bottom right $(r\x m)$-block vanishes.

\begin{theorem}
\label{th:ders}
Suppose the controller (\ref{xi_eta}), (\ref{abce}) is admissible, and the closed-loop system (\ref{dX_cAB})  is driven by vacuum fields and is in the invariant Gaussian quantum state.
Then the partial Frechet derivatives of the QEF rate (\ref{Ups}) for the system with respect to the energy and coupling matrices $R_2$, $M_2$, $L_2$  of the controller can be computed as
\begin{align}
\label{dUpsdR2der}
    \d_{R_2}
    \Ups_\theta
    & =
    -2
        \bS(\Theta_2 \d_a \Ups_\theta),\\
\label{dUpsdM2der}
    \d_{M_2}
    \Ups_\theta
    & =
    2
    (
        2J_2 M_2\bA(\Theta_2 \d_a \Ups_\theta)
        +
        \d_b \Ups_\theta^\rT\Theta_2
        -
        J_2 d^\rT \d_c \Ups_\theta
    ),\\
\label{dUpsdL2der}
    \d_{L_2}
    \Ups_\theta
    & =
    2
    (
        2\wt{J}_1 L_2\bA(\Theta_2 \d_a \Ups_\theta)
        +
        \d_e \Ups_\theta^\rT\Theta_2)
\end{align}
(with $\bS(\mu):= \frac{1}{2}(\mu+\mu^\rT)$ and $\bA(\mu):= \frac{1}{2}(\mu-\mu^\rT)$ the symmetrizer and antisymmetrizer of square matrices). Here, the partial Frechet derivatives of $\Ups_\theta$ with respect to the controller matrices $a$, $b$, $c$, $e$ (as independent variables) are the blocks of the matrix 
\begin{equation}
\label{mat}
    \theta \Gamma_1^\rT \chi_\theta \Gamma_2^\rT
    =
    \begin{bmatrix}
      \d_a\Ups_\theta & \d_b\Ups_\theta & \d_e\Ups_\theta \\
      \d_c\Ups_\theta & * & *
    \end{bmatrix},
\end{equation}
where the blocks ``$*$'' are irrelevant, and the matrices $\Gamma_1$, $\Gamma_2$, $\chi_\theta$ are given by (\ref{GGG}), (\ref{chitheta}).
\end{theorem}

\begin{proof}
In view of the relation $\delta \ln\det \mu = \bra \mu^{-1}, \delta\mu\ket$ for positive definite Hermitian  matrices $\mu$ (with $\bra \alpha, \beta\ket:= \Tr(\alpha^* \beta)$ the Frobenius inner product \cite{HJ_2007} of identically dimensioned real or complex matrices $\alpha$, $\beta$),  it follows from (\ref{UpsDel}) that, for a given $\theta$, the first variation of the QEF rate with respect to the admissible controller matrices can be represented as
\begin{equation}
\label{dUps}
  \delta \Ups_\theta
  =
  \frac{\theta}{4\pi}
  \int_{\mR}
  \bra
    \Delta_\theta(\lambda)^{-1},
    \delta \Sigma_\theta(\lambda)
  \ket
  \rd \lambda,
\end{equation}
where use is also made of $\delta \Delta_\theta = -\theta \delta \Sigma_\theta$ due to (\ref{Del}).
For a fixed but otherwise arbitrary frequency $\lambda\in \mR$,   the first variation of the matrix $\Sigma_\theta(\lambda)$ in (\ref{Sig}) is computed as
\begin{equation}
\label{dSig}
    \delta \Sigma_\theta
     =
    S(
    F^* \phi(2i\theta \Psi) \delta F
    +
    (\delta F^*) \phi(2i\theta \Psi) F
    +
    F^* (\delta\phi(2i\theta \Psi)) F
    ) S,
\end{equation}
since the matrix $S$ in (\ref{Oroot}) is constant.
By substituting (\ref{dSig}) into (\ref{dUps}), the integrand  takes the form
\begin{align}
\nonumber
  \bra
    \Delta_\theta^{-1},
    \delta \Sigma_\theta
  \ket
  &=
  \bra
    S \Delta_\theta^{-1} S,
    F^* \phi(2i\theta \Psi) \delta F
    +
    (\delta F^*) \phi(2i\theta \Psi) F
    +
    F^* (\delta\phi(2i\theta \Psi)) F
  \ket  \\
\nonumber
  & =
  \bra
    S \Delta_\theta^{-1} S,
    F^* \phi(2i\theta \Psi) \delta F
    +
    (F^* \phi(2i\theta \Psi) \delta F)^*
  \ket
    +
  \bra
    S \Delta_\theta^{-1} S,
    F^* (\delta\phi(2i\theta \Psi)) F
  \ket  \\
\nonumber
  & =
  2\Re
  \bra
    S \Delta_\theta^{-1} S,
    F^* \phi(2i\theta \Psi) \delta F
  \ket
    +
  \bra
    F S \Delta_\theta^{-1} S F^*,
    \delta\phi(2i\theta \Psi)
  \ket  \\
\nonumber
  & =
  2\Re
    \bra
        \phi(2i\theta \Psi) F
        S \Delta_\theta^{-1} S,
        \delta F
    \ket
     +
    \bra
        F S \Delta_\theta^{-1} S F^*,
        \delta\phi(2i\theta \Psi)
  \ket  \\
\label{braket}
    & =
    2\Re
    \bra
        \phi_\theta,
        \delta F
    \ket
    +
    \bra
        F S \Delta_\theta^{-1} S F^*,
        \delta\phi(2i\theta \Psi)
  \ket,
\end{align}
where we have used (\ref{phitheta_psitheta}) and the Hermitian property of the matrices $S$, $\Delta_\theta$, $\phi(2i\theta \Psi)$. Here, use is also made of the identity $\bra \alpha, \sigma \beta\tau\ket = \bra \sigma^* \alpha \tau^*, \beta\ket $ for appropriately dimensioned complex matrices $\alpha$, $\beta$, $\sigma$, $\tau$ (that is, the representation
\begin{equation}
\label{sandadj}
    \[[[\sigma, \tau\]]]^\dagger = \[[[\sigma^*, \tau^*\]]]
\end{equation}
for the adjoint of the sandwich operator (\ref{sand})),   along with $\bra \alpha, \beta + \beta^*\ket = 2\Re \bra \alpha, \beta\ket$ in the case when $\alpha$ is Hermitian. Application of (\ref{fab0}) from Lemma~\ref{lem:fmat'} (see Appendix~\ref{sec:matfun} and references therein)
to the function $\phi$ in (\ref{phi}) on the Hilbert space $\mH_r$ (with the Frobenius inner product $\bra\cdot, \cdot\ket$) leads to
\begin{equation}
\label{dbraket}
    \bra
        F S \Delta_\theta^{-1} S F^*,
        \delta\phi(2i\theta \Psi)
    \ket
    =
    2i\theta
    \bra
        \d_\alpha
        \phi(\alpha)|_{\alpha=2i\theta\Psi}
        (F S \Delta_\theta^{-1} S F^*),
        \delta\Psi
    \ket
    =
    2i\theta
    \bra
        \varpi_\theta,
        \delta\Psi
    \ket,
\end{equation}
where the Frechet derivative $\d_\alpha\phi(\alpha)$ of $\phi$ is a linear operator acting on the space $\mH_r$ as
\begin{equation}
\label{dphi}
    \d_\alpha
    \phi(\alpha)(\beta)
     =
    \left(
    \phi
    \left(
    \begin{bmatrix}
      \alpha & 0\\
      \beta & \alpha
    \end{bmatrix}
    \right)
    \right)_{21}
    \in \mH_r,
    \qquad
    \alpha, \beta \in \mH_r,
\end{equation}
with $(\cdot)_{21}$ denoting the bottom left $( r\x  r)$-block  of the $(2r\x 2r)$-matrix.  In (\ref{dbraket}), the relation (\ref{dphi}) is applied to the matrices $\alpha:= 2i\theta \Psi$ and $\beta:=F S \Delta_\theta^{-1} S F^*$ (which are both in $\mH_r$), thereby leading to the matrix $\varpi_\theta$ in (\ref{varpitheta}).  The first variation of $\Psi$ in (\ref{Phi0_Psi0}) is computed as
$    \delta \Psi
    =
    (\delta F) J F^* + FJ\delta F^*
$, and its substitution into (\ref{dbraket}) yields
\begin{align}
\nonumber
    \bra
        F S \Delta_\theta^{-1} S F^*,
        \delta\phi(2i\theta \Psi)
    \ket
    & =
    2i\theta
    \bra
        \varpi_\theta,
        (\delta F) J F^* + FJ\delta F^*
    \ket\\
\nonumber
    & =
    2i\theta
    \bra
        \varpi_\theta,
        (\delta F) J F^* -((\delta F) J F^*)^*
    \ket
    =
    -4 \theta
    \Im
    \bra
        \varpi_\theta,
        (\delta F) J F^*
    \ket\\
\label{dbraket1}
    &
    =
    4 \theta
    \Im
    \bra
        \varpi_\theta F J,
        \delta F
    \ket
    =
    4\theta
    \Im
    \bra
        \psi_\theta,
        \delta F
    \ket
    =
    4\theta
    \Re
    \bra
        i\psi_\theta,
        \delta F
    \ket,
\end{align}
where use is made of (\ref{phitheta_psitheta}), (\ref{varpitheta}),  the antisymmetry of the matrix $J$ in (\ref{J}),  and the matrix identity $\bra \alpha, \beta - \beta^*\ket = 2i\Im \bra \alpha, \beta\ket$ with $\alpha$ Hermitian.  In turn, the first variation of the transfer matrix $F$ in (\ref{F0_G}) (as a function of the matrices  $\cA$, $\cB$, $\cC$) is represented as
\begin{equation}
\label{dF}
  \delta F
  =
  \cC G (\delta \cA) G \cB + \cC G \delta \cB + (\delta \cC) G \cB
  =
    \begin{bmatrix}
      \cC G & I_r
    \end{bmatrix}
    (\delta \Gamma )
    \begin{bmatrix}
      G \cB \\
      I_m
    \end{bmatrix},
\end{equation}
due to the
relation $\delta G = G(\delta \cA)G$,  which is obtained from the second equality in  (\ref{F0_G})  by using the identity $\delta (\mu^{-1}) = -\mu^{-1} (\delta \mu) \mu^{-1}$ for nonsingular matrices $\mu$. Here,
\begin{equation}
\label{dABC}
    \delta \Gamma
    =
  \begin{bmatrix}
    \delta\cA & \delta\cB\\
    \delta\cC & 0
  \end{bmatrix},
\end{equation}
in accordance with the first equality in (\ref{Gamgam}). 
By combining (\ref{braket}) with  (\ref{dbraket1}) and using (\ref{dF}), it follows that
\begin{align}
\nonumber
  \bra
    \Delta_\theta^{-1},
    \delta \Sigma_\theta
  \ket
    & =
    2\Re
    \bra
        \phi_\theta,
        \delta F
    \ket
    +
    4\theta
    \Re
    \bra
        i\psi_\theta,
        \delta F
    \ket
    =
    2
    \Re
    \bra
        \phi_\theta + 2i\theta \psi_\theta,
        \delta F
    \ket    \\
\nonumber
    & =
    2\Re
    \Bra
    \begin{bmatrix}
      G^* \cC^\rT \\
      I_r
    \end{bmatrix}
    (\phi_\theta + 2i\theta \psi_\theta)
    \begin{bmatrix}
      \cB^\rT G^* &I_m
    \end{bmatrix},
    \delta\Gamma
    \Ket\\
\label{dbraket2}
    & =
    2
    \Bra
    \Re
    \fP
    \left(
    \begin{bmatrix}
      G^* \cC^\rT \\
      I_r
    \end{bmatrix}
    (\phi_\theta + 2i\theta \psi_\theta)
    \begin{bmatrix}
      \cB^\rT G^* &I_m
    \end{bmatrix}
    \right),
    \delta\Gamma
    \Ket,
\end{align}
where $\fP$ is the operator (\ref{fP}), 
and the matrix identity $\Re \bra \alpha, \beta\ket= \bra \Re \alpha, \beta\ket $ with $\beta$ real is applied to $\beta := \delta\Gamma$. Since $\delta \Gamma$ in (\ref{dABC}) does not depend on the frequency $\lambda$, then by integrating both sides of (\ref{dbraket2}) over $\lambda \in \mR$ and using (\ref{chitheta}), the  first variation (\ref{dUps}) acquires the form
\begin{equation}
\label{dUps3}
  \delta \Ups_\theta
  =
  \theta
  \bra
    \chi_\theta,
    \delta\Gamma
  \ket,
\end{equation}
which yields the Frechet derivative of the QEF rate with respect to the matrix $\Gamma \in \fT_{n+\nu,m,r}$ in (\ref{Gamgam}):
\begin{equation}
\label{dUpsdABC}
    \d_\Gamma\Ups_\theta
    =
    \begin{bmatrix}
      \d_{\cA} \Ups_\theta & \d_{\cB} \Ups_\theta\\
        \d_{\cC} \Ups_\theta & 0
    \end{bmatrix}
    =
    \theta \chi_\theta.
\end{equation}
From the affine dependence of $\Gamma$ on $\gamma$ in (\ref{Gamgam}),  which describes the second map in (\ref{comp}), it follows that
\begin{equation}
\label{dABC1}
    \delta \Gamma
  =
  \Gamma_1
  (\delta \gamma)
  \Gamma_2,
  \qquad
  \delta \gamma
    =
    \begin{bmatrix}
      \delta a & \delta b & \delta e \\
      \delta c & 0 & 0
    \end{bmatrix},
\end{equation}
in accordance with (\ref{dGamdgam}), with the matrices $\Gamma_1$, $\Gamma_2$ given by  (\ref{GGG}). Therefore, by an operator version of the chain rule for differentiating composite functions,
substitution of the  first equality from (\ref{dABC1}) into (\ref{dUps3}) yields
$  \delta \Ups_\theta
  =
  \theta
  \bra
    \Gamma_1^\rT \chi_\theta \Gamma_2^\rT,
    \delta\gamma
  \ket
  =
  \theta
  \bra
    \daleth_{\nu,m_2+p_1,p_2}(\Gamma_1^\rT \chi_\theta \Gamma_2^\rT),
    \delta\gamma
  \ket
$, 
whereby
\begin{equation}
\label{dUpsdgamma}
    \d_\gamma\Ups_\theta
    =
    \theta \daleth_{\nu,m_2+p_1,p_2}
    (\Gamma_1^\rT \chi_\theta \Gamma_2^\rT)
\end{equation}
is the Frechet derivative of $\Ups_\theta$ with respect to the matrix $\gamma \in \fT_{\nu,m_2+p_1,p_2}$ in (\ref{Gamgam}). The relation (\ref{dUpsdgamma}) allows the partial Frechet derivatives of $\Ups_\theta$ with respect to the controller matrices $a$, $b$, $c$, $e$ (as if they were  independent variables)   to be  recovered as the corresponding blocks 
of the matrix (\ref{mat}). 
It now remains to take into account the first map in (\ref{comp}) specified by (\ref{abce}). The latter leads to the first variations of the QEF rate $\Ups_\theta$ with respect to the energy and coupling matrices $R_2$, $M_2$, $L_2$ of the controller:
\begin{align}
\label{dUpsdR2}
    \delta_{R_2}
    \Ups_\theta
    & =
    \bra
        \d_a \Ups_\theta,
        2\Theta_2 \delta R_2
    \ket
    =
    -2
    \bra
        \Theta_2 \d_a \Ups_\theta,
        \delta R_2
    \ket
    =
    -2
    \bra
        \bS(\Theta_2 \d_a \Ups_\theta),
        \delta R_2
    \ket,\\
\nonumber
    \delta_{M_2}
    \Ups_\theta
    & =
    \bra
        \d_a \Ups_\theta,
        2\Theta_2 (M_2^\rT J_2 \delta M_2 + (\delta M_2^\rT) J_2 M_2)
    \ket
    +
    \bra
        \d_b \Ups_\theta,
        2\Theta_2 \delta M_2^\rT
    \ket
    +
    \bra
        \d_c \Ups_\theta,
        2d J_2 \delta M_2
    \ket    \\
\nonumber
    & =
    -2
    \bra
        \bA(\Theta_2 \d_a \Ups_\theta),
        M_2^\rT J_2 \delta M_2 -(M_2^\rT J_2 \delta M_2)^\rT
    \ket
    -
    2
    \bra
        \d_b \Ups_\theta^\rT,
        \delta M_2 \Theta_2
    \ket
    -
    2
    \bra
        J_2 d^\rT \d_c \Ups_\theta,
        \delta M_2
    \ket    \\
\label{dUpsdM2}
    & =
    2
    \bra
        2J_2 M_2\bA(\Theta_2 \d_a \Ups_\theta)
        +
        \d_b \Ups_\theta^\rT\Theta_2
        -
        J_2 d^\rT \d_c \Ups_\theta,
        \delta M_2
    \ket,\\
\nonumber
    \delta_{L_2}
    \Ups_\theta
    & =
    \bra
        \d_a \Ups_\theta,
        2\Theta_2 (L_2^\rT \wt{J}_1 \delta L_2 + (\delta L_2^\rT) \wt{J}_1 L_2)
    \ket+
    \bra
        \d_e \Ups_\theta,
        2\Theta_2 \delta L_2^\rT
    \ket\\
\label{dUpsdL2}
    & =
    2
    \bra
        2\wt{J}_1 L_2\bA(\Theta_2 \d_a \Ups_\theta)
        +
        \d_e \Ups_\theta^\rT\Theta_2,
        \delta L_2
    \ket,
\end{align}
where use is also made of the symmetry of $R_2$ and antisymmetry of the CCR matrices $\Theta_2$, $J_2$, $\wt{J}_1$ along with the orthogonality of the subspaces of real symmetric and real antisymmetric matrices. By (\ref{dUpsdR2})--(\ref{dUpsdL2}), the corresponding partial Frechet derivatives take  the form (\ref{dUpsdR2der})--(\ref{dUpsdL2der}).
\end{proof}

The proof of Theorem~\ref{th:ders} shows that the specific form (\ref{Ups}) (or  (\ref{UpsDel})) of the cost functional $\Ups_\theta$ enters its partial Frechet derivatives (\ref{dUpsdR2der})--(\ref{mat}) only through the  matrix $\chi_\theta$ in (\ref{chitheta}). Indeed, the maps $    \cE
    \mapsto
    (a,b,c,e)
    \mapsto
    (\cA,\cB,\cC)
$  in (\ref{comp}) remain unchanged for different cost functionals using the same criterion process (\ref{Z_U_cC}). For this reason, the matrix $\chi_\theta$, which, in view of (\ref{dUpsdABC}),  captures the Frechet derivatives of the  function
$    (\cA,\cB,\cC)
    \mapsto
    \Ups_\theta
$,   will be referred to as the \emph{core matrix} for the QEF rate.

In the limit, as $\theta \to 0+$, the  functions (\ref{Del}),  (\ref{phitheta_psitheta}) reduce to $\Delta_0 = I_m$, $\phi_0 = F\Omega^\rT$, $\psi_0= \frac{1}{2}F\Omega^\rT F^* FJ$, whose substitution into (\ref{chitheta}) at $\theta=0$ yields the matrix
\begin{equation}
 \label{chi0}
    \chi_0
    =
  \frac{1}{2\pi}
  \int_{\mR}
  \Re
  \fP
    \left(
    \begin{bmatrix}
      G^* \cC^\rT \\
      I_r
    \end{bmatrix}
    F\Omega^\rT
    \begin{bmatrix}
      \cB^\rT G^* &I_m
    \end{bmatrix}
      \right)
      \rd \lambda
 \end{equation}
 (note that $S^2 = \Omega^\rT$ in view of (\ref{Oroot})). Similarly to (\ref{dUpsdABC}), the matrix $\chi_0$ can be represented as
\begin{equation}
\label{PQH}
    \chi_0
    =
    \d_\Gamma
    \Ups_*
    =
    \begin{bmatrix}
      \d_{\cA} \Ups_* & \d_{\cB} \Ups_*\\
        \d_{\cC} \Ups_* & 0
    \end{bmatrix}
    =
    \begin{bmatrix}
      \cH & \cQ \cB\\
      \cC \cP  & 0
    \end{bmatrix}
\end{equation}
and plays the role of a core matrix
associated with the following mean square cost for the same  criterion process (\ref{Z_U_cC}) of the closed-loop system in the invariant Gaussian state:
\begin{equation}
\label{Ups*}
    \Ups_*
    :=
    \frac{1}{2}
    \bE (Z(0)^\rT Z(0))
    =
    \lim_{\theta\to 0+}
    \frac{\Ups_\theta}{\theta}
    =
    \frac{1}{2T} \bE Q_T,
\end{equation}
which holds for any $T>0$ (in the invariant state) and is the growth rate for the right-hand side of (\ref{Xilim}).
The last equality in (\ref{PQH}) employs
the Hankelian $\cH:= \cQ\cP$ of the system (see, for example,  \cite[Lemma~2]{VP_2013a} or \cite[Lemma 7 of Appendix B]{VP_2010}) associated with the  controllability Gramian $\cP$ from (\ref{cP}) and the observability Gramian $\cQ = \bL_{\cA^\rT}(\Pi)$ of the pair $(\cA, \cC)$ satisfying the ALE $\cA^\rT \cQ + \cQ \cA + \Pi  = 0$, where
use is made of an auxiliary matrix
\begin{equation}
\label{CC}
  \Pi
  :=
  \cC^\rT \cC
  \in
  \mS_{n+\nu}^+.
\end{equation}
The right-hand side of (\ref{PQH}) can also be obtained directly by evaluating (\ref{chi0}) through the Plancherel theorem applied to the transfer functions (\ref{F0_G}) and their complex conjugate transpose.

We will now return to the general case of $\theta>0$.  In combination with (\ref{GGG}), (\ref{phitheta_psitheta})--(\ref{chitheta}),  (\ref{mat}) (with the closed-loop system matrices $\cA$, $\cB$, $\cC$ being computed according to (\ref{dX_cAB}), (\ref{Z_U_cC}) and the parameterization (\ref{abce}) of the controller matrices), the equations
\begin{align}
\label{R2opt}
        \bS(\Theta_2 \d_a \Ups_\theta)
        & = 0,\\
\label{M2opt}
        2J_2 M_2\bA(\Theta_2 \d_a \Ups_\theta)
        +
        \d_b \Ups_\theta^\rT\Theta_2
        -
        J_2 d^\rT \d_c \Ups_\theta
        & = 0, \\
\label{L2opt}
        2\wt{J}_1 L_2\bA(\Theta_2 \d_a \Ups_\theta)
        +
        \d_e \Ups_\theta^\rT\Theta_2
        & = 0,
\end{align}
obtained by equating the partial Frechet derivatives (\ref{dUpsdR2der})--(\ref{dUpsdL2der}) to zero, provide first-order necessary conditions of optimality in frequency domain for the quantum risk-sensitive control problem (\ref{opt}) in the class of admissible coherent quantum controllers (\ref{xi_eta}), (\ref{abce}) for the quantum plant (\ref{x_y}).
The equations (\ref{R2opt})--(\ref{L2opt}) can be solved numerically by using the Frechet derivatives for a gradient descent iterative algorithm in the space $\fE$  of the matrix triples $\cE$ from (\ref{comp}), similarly to \cite{SVP_2017} for the coherent quantum LQG (CQLQG) control problem $\Ups_* \to \inf$ with the mean square cost (\ref{Ups*}). Such an algorithm can be initialized, for example, at a solution of the CQLQG control problem, which involves the corresponding core matrix $\chi_0$ from (\ref{chi0}), (\ref{PQH})  and can also be obtained by using the homotopy and discounting ideas \cite{VP_2021}. At every step,  the gradient descent requires the evaluation of the core matrix $\chi_\theta$ in (\ref{chitheta}), which is the crucial part of this approach. To this end, Section~\ref{sec:statespace} will outline a state-space computation of $\chi_\theta$ using spectral factorizations and an  infinite cascade of classical linear systems. In addition to the role for numerical optimization, the general structure of the core matrix $\chi_\theta$ in (\ref{chitheta}) (and its limiting case $\chi_0$ in   (\ref{chi0})) suggests a link between the risk-sensitive and CQLQG control problems,  which is discussed in the next section.

\section{Infinitesimal equivalence to a weighted CQLQG control problem}
\label{sec:infeq}

From (\ref{phitheta_psitheta}), (\ref{varpitheta}),  it follows that, at any given frequency $\lambda \in \mR$, the intermediate factor  $\phi_\theta + 2i\theta \psi_\theta$ of the integrand in the core matrix (\ref{chitheta}) can be represented as  the image
\begin{equation}
\label{fMF}
    \phi_\theta(\lambda) + 2i\theta \psi_\theta(\lambda)
     =
     \fM_{\theta,F}(F(i\lambda))
\end{equation}
of the closed-loop system transfer matrix $F(i\lambda)$ in (\ref{F0_G}) under a multisandwich operator on $\mC^{r\x m}$:
\begin{equation}
\label{fM}
    \fM_{\theta, F}
    :=
    \[[[
        \phi(2i\theta\Psi(\lambda)), S\Delta_\theta(\lambda)^{-1} S
        \mid
        \varpi_\theta(\lambda), 2i\theta J
    \]]] ,
\end{equation}
whose dependence on $\lambda$ is omitted for brevity. In view of (\ref{sandadj}), the operator $\fM_{\theta,F}$ is self-adjoint since all the four matrices, which specify it, are Hermitian. Note that the operator $\fM_{\theta,F}$ itself depends on $F$ through $\Psi$, $\Delta_\theta$, $\varpi_\theta$.

Now,  consider a generalized version of the CQLQG control problem with a weight\-ed quadratic cost functional
\begin{equation}
\label{V}
    V
    :=
    \frac{1}{4\pi}
    \int_{\mR}
    \bra
        F,
        \fM(F)
    \ket
    \rd \lambda
    \to \inf
\end{equation}
for the same criterion process (\ref{Z_U_cC}).
Here, 
$\fM$ is a given frequency-dependent self-adjoint operator on the Hilbert space $\mC^{r\x m}$ (with the Frobenius inner product). In order to guarantee the convergence of the integral in (\ref{V}) for strictly proper stable transfer functions $F$, we assume that the  operator norm of $\fM$ is bounded uniformly  over the frequency:
\begin{equation}
\label{fMnorm}
    \sn \fM \sn
    :=
    \sup_{\lambda \in \mR}
    \|\fM\| < +\infty,
\end{equation}
so that
$
    |V| \<
    \frac{1}{2}
    \sn \fM \sn
    \|F\|_2^2
    < +\infty
$,
where $\|F\|_2 := \sqrt{\frac{1}{2\pi} \int_\mR \|F(i\lambda)\|_\rF^2\rd \lambda}$ is the norm of $F$ in the Hardy space $\cH_2$ defined in terms of the Frobenius norm $\|\cdot\|_\rF:= \sqrt{\bra \cdot, \cdot\ket}$ of matrices.

\begin{lemma}
\label{lem:core}
For a given frequency-dependent self-adjoint  operator $\fM$ on $\mC^{r\x m}$  satisfying (\ref{fMnorm}),   the weighted CQLQG control problem (\ref{V}) over stabilizing coherent quantum controllers (\ref{xi_eta}), (\ref{abce}) has the core matrix
\begin{equation}
\label{chi}
    \chi
    :=
    \d_\Gamma
    V
    =
    \begin{bmatrix}
      \d_{\cA} V & \d_{\cB} V\\
        \d_{\cC} V & 0
    \end{bmatrix}
    =
  \frac{1}{2\pi}
  \int_{\mR}
  \Re
  \fP
  \left(
    \begin{bmatrix}
      G^* \cC^\rT \\
      I_r
    \end{bmatrix}
    \fM(F)
    \begin{bmatrix}
      \cB^\rT G^* &I_m
    \end{bmatrix}
      \right)
      \rd \lambda,
\end{equation}
where $\fP$ is the projection operator in (\ref{fP}), and use is made of the closed-loop system matrices $\cA$, $\cB$, $\cC$  from (\ref{dX_cAB}), (\ref{Z_U_cC}) assembled into the matrix $\Gamma$ in (\ref{Gamgam}),  along with the transfer function $G$ from (\ref{F0_G}).
\end{lemma}
\begin{proof}
Since the operator $\fM$ itself does not depend on the transfer function $F$, then at any frequency $\lambda\in \mR$, the first variation of the integrand in (\ref{V}) with respect to $F$ is computed as
\begin{equation}
\label{dFMF}
    \delta
    \bra
        F,
        \fM(F)
    \ket
     =
    \bra
        F,
        \fM(\delta F)
    \ket
    +
    \bra
        \delta F,
        \fM(F)
    \ket
    =
    \bra
        \fM(F),
        \delta F
    \ket
    +
    \overline{
    \bra
        \fM(F),
        \delta F
    \ket}
    =
    2\Re
    \bra
        \fM(F),
        \delta F
    \ket
\end{equation}
due to the self-adjointness of $\fM$. By (\ref{dFMF}), the first variation of the cost functional $V$ takes the form
\begin{equation}
\label{dV}
    \delta V
    =
    \frac{1}{4\pi}
    \int_\mR
    \delta
    \bra
        F,
        \fM(F)
    \ket
    \rd \lambda
    =
    \frac{1}{2\pi}
    \int_\mR
    \Re
    \bra
        \fM(F),
        \delta F
    \ket
    \rd \lambda.
\end{equation}
Similarly to (\ref{dbraket2}), substitution of (\ref{dF}), (\ref{dABC}) into the last integrand in (\ref{dV}) leads to
$
    \Re
    \bra
        \fM(F),
        \delta F
    \ket
    =
    \Bra
    \Re
    \fP
    \left(
    \begin{bmatrix}
      G^* \cC^\rT \\
      I_r
    \end{bmatrix}
    \fM(F)
    \begin{bmatrix}
      \cB^\rT G^* &I_m
    \end{bmatrix}
    \right),
    \delta\Gamma
    \Ket
$.
Hence, $\delta V = \bra \chi, \delta \Gamma\ket$, with the matrix $\chi$ in (\ref{chi}) indeed being the core matrix associated with the cost functional $V$.
\end{proof}

The following theorem, which is a corollary of Lemma~\ref{lem:core},  establishes a connection between the coherent quantum risk-sensitive  control problem (\ref{opt}) and the weighted  CQLQG control problem (\ref{V}) at the level of the first-order necessary conditions of optimality.

\begin{theorem}
\label{th:infeq}
For a given $\theta>0$, suppose $F_*$ is the transfer function of the closed-loop system with an admissible controller (\ref{xi_eta}), (\ref{abce}), and the operator
\begin{equation}
\label{fM*}
    \fM:= \fM_{\theta,F_*}
\end{equation}
is associated with $F_*$ as in (\ref{fM}) and fixed. Then the controller is a stationary point of the  coherent quantum risk-sensitive  control problem (\ref{opt}) in the sense of (\ref{R2opt})--(\ref{L2opt}) if and only if it is so for the weighted  CQLQG control problem (\ref{V}) with the weighting operator (\ref{fM*}).
\end{theorem}
\begin{proof}
Substitution of (\ref{fMF}) into (\ref{chitheta}) shows that the resulting core matrix $\chi_\theta$ of the problem (\ref{opt}), computed for the closed-loop system $F_*$,  coincides with the core matrix $\chi$ in  (\ref{chi}) for the problem (\ref{V}) with the weighting operator $\fM$ in (\ref{fM*}). The relation $\chi_\theta = \chi$ implies the equality $\d_\cE\Ups_\theta = \d_\cE V$ for the corresponding Frechet derivatives $\d_\cE(\cdot) := (\d_{R_2}(\cdot),  \d_{M_2}(\cdot),  \d_{L_2}(\cdot))$ of the cost functionals $\Ups_\theta$, $V$   with respect to the triple $\cE$ from (\ref{comp}) at such a controller. Hence, the latter satisfies the stationarity condition $\d_\cE\Ups_\theta = 0$ if and only if it does $\d_\cE  V = 0$. The boundedness of the operator (\ref{fM*}) is addressed in Lemma~\ref{lem:bound}.
\end{proof}

In the context of the problem  (\ref{V}), the above proof is completed by the following lemma which clarifies the correctness of the right-hand side of (\ref{fM*}) as a bounded operator.

\begin{lemma}
\label{lem:bound}
For any given $\theta>0$ and an admissible controller (\ref{xi_eta}), (\ref{abce}), the operator $\fM_{\theta, F}$ associated with the corresponding closed-loop system transfer function $F$ by (\ref{fM}),  has a finite norm (\ref{fMnorm}):
\begin{equation}
\label{fMnorm0}
    \sn\fM_{\theta, F}\sn
    \<
    \frac{2\vartheta_r(2\theta \|F\|_\infty^2)}{\inf_{\lambda \in\mR}\lambda_{\min}(\Delta_\theta(\lambda))},
    \qquad
    \vartheta_r(u)
    :=
    \phi(u) + \sqrt{r} \phi'(u) u,
\end{equation}
where $\phi'$ is the derivative of the function $\phi$ from (\ref{phi}), and $\|F\|_\infty = \sup_{\lambda\in \mR}\|F(i\lambda)\|$ is the $\cH_\infty$-norm of $F$.
\end{lemma}
\begin{proof}
At any frequency $\lambda \in \mR$, the operator norm of the multisandwich operator (\ref{fM}) admits an upper bound
\begin{align}
\nonumber
    \|\fM_{\theta, F}\|
    & \<
    \|
    \[[[
        \phi(2i\theta\Psi(\lambda)), S\Delta_\theta(\lambda)^{-1} S
    \]]]
    \|
    +
    \|
    \[[[
        \varpi_\theta(\lambda), 2i\theta J
    \]]]
    \|\\
\nonumber
    & \<
    \|\phi(2i\theta\Psi(\lambda))\|
    \|S\Delta_\theta(\lambda)^{-1} S\|
    +
    \|\varpi_\theta(\lambda)\|
    \|2i\theta J\|\\
\nonumber
    & \<
    \phi(2\theta\|\Psi(\lambda)\|)
    \lambda_{\max}
    (S\Delta_\theta(\lambda)^{-1} S)
    +
    2\theta
    \|\varpi_\theta(\lambda)\|
    \|J\|\\
\label{fMnorm1}
    & \<
    \phi(2\theta\|\Psi(\lambda)\|)
    \frac{\lambda_{\max}(S^2)}{\lambda_{\min}(\Delta_\theta(\lambda))}
    +
    2\theta
    \|\varpi_\theta(\lambda)\|
    =
    \frac{2    \phi(2\theta\|\Psi(\lambda)\|)}{\lambda_{\min}(\Delta_\theta(\lambda))}
    +
    2\theta
    \|\varpi_\theta(\lambda)\|,
\end{align}
where use is made of (\ref{sbound}) and  positivity of the coefficients $\phi_k$   in the expansion of $\phi$   in (\ref{phi}) along with the unitarity of the matrix $J$ in (\ref{J}) (whereby $\|J\| = 1$). Also, we have employed (\ref{Oroot}), the fact that $\lambda_{\max}(\Omega^\rT) = 2$, and the relation (\ref{unipos}). A combination of the inequalities $\|N\| \< \|N\|_\rF \< \sqrt{r} \|N\|$ for the operator and Frobenius norms of any matrix $N \in \mC^{r\x r}$ with the norm bound (\ref{dfdanorm1}) applied to the Frechet derivative in (\ref{varpitheta}) leads to
\begin{align}
\nonumber
    \|\varpi_\theta(\lambda)\|
    & \< \|\varpi_\theta(\lambda)\|_\rF
    \<
    \|\d_\alpha \phi(\alpha)|_{\alpha = 2i\theta\Psi(\lambda)}\|
    \|F(i\lambda) S\Delta_\theta(\lambda)^{-1} SF(i\lambda)^*\|_\rF \\
\nonumber
    & \<
    \sqrt{r}
    \phi'(\|2i\theta\Psi(\lambda)\|)
    \|F(i\lambda) S\Delta_\theta(\lambda)^{-1} SF(i\lambda)^*\|\\
\label{varpinorm}
    & \<
    \sqrt{r}
    \phi'(2\theta \|\Psi(\lambda)\|)
    \frac{\lambda_{\max}(F(i\lambda) \Omega^\rT F(i\lambda)^*)}{\lambda_{\min}(\Delta_\theta(\lambda))}
    \<
    2\sqrt{r}
    \phi'(2\theta \|\Psi(\lambda)\|)
    \frac{\|F(i\lambda)\|^2}{\lambda_{\min}(\Delta_\theta(\lambda))}.
\end{align}
By substituting (\ref{varpinorm}) into (\ref{fMnorm1}) and taking the supremum over $\lambda \in \mR$, it follows that
$$
    \sn\fM_{\theta, F}\sn
    \<
    2
    \sup_{\lambda \in \mR}
    \frac{\phi(2\theta\|\Psi(\lambda)\|)+ 2\theta
    \sqrt{r}
    \phi'(2\theta \|\Psi(\lambda)\|)\|F(i\lambda)\|^2}{\lambda_{\min}(\Delta_\theta(\lambda))},
$$
which leads to (\ref{fMnorm0}) in view of (\ref{unipos}), the inequality $\|\Psi(\lambda)\| \< \|J\| \|F(i\lambda)\| \< \|F\|_\infty^2$ for the function $\Psi$ in  (\ref{Phi0_Psi0}),  and the property that both $\phi(u)$ and $\phi'(u)$ are increasing functions of $u\> 0$.
\end{proof}

The assertion of Theorem~\ref{th:infeq} can be interpreted as ``infinitesimal equivalence'' between the coherent quantum control problems with the QEF and weighted quadratic performance criteria. Furthermore, if there is an efficient way of solving the class of weighted CQLQG control problems (\ref{V}) through finding their stationary points,  then these solutions can be used at every step of an iterative procedure for solving the coherent quantum risk-sensitive control problem (\ref{opt}),  which starts from an initial  approximation $F_0$ and  produces a sequence of closed-loop system transfer functions $F_1, F_2, F_3, \ldots$.  More precisely, the $k$th step of this procedure solves the  weighted CQLQG control problem
\begin{equation}
\label{Vk}
    V_k
    :=
    \frac{1}{4\pi}
    \int_{\mR}
    \bra
        F,
        \fM_{\theta, F_{k-1}}(F)
    \ket
    \rd \lambda
    \to \inf,
    \qquad
    k \> 1,
\end{equation}
over the closed-loop system transfer functions $F$ resulting from   stabilizing coherent quantum controllers (\ref{xi_eta}), (\ref{abce}), where a fixed weighting operator $\fM:= \fM_{\theta, F_{k-1}}$ is specified by the transfer function $F_{k-1}$ obtained at the previous step. Accordingly, a solution $F_k:=F$ of (\ref{Vk}) provides the next element of the sequence. If the parameter triples $\cE_k:= (R_2^{(k)}, M_2^{(k)}, L_2^{(k)})$ of the corresponding controllers have a limit $\cE_*:= \lim_{k\to +\infty} \cE_k$ which  describes an admissible controller, with the closed-loop system transfer function $F_*$ being an appropriate limit of $(F_k)_{k\> 0}$,
then $\cE_*$ belongs to the set of stationary points  of the weighted CQLQG control problem (\ref{V}) with the ``frozen'' weighting operator (\ref{fM*}).  On the other hand, in view of by Theorem~\ref{th:infeq}, this set coincides with that of stationary points of the coherent quantum risk-sensitive control problem (\ref{opt}).

The above procedure is more complicated than the gradient descent (mentioned in Section~\ref{sec:opt}) for the QEF rate $\Ups_\theta$ as a function of the matrix triple $\cE$ in (\ref{comp}). However, irrespective of their complexity, both approaches require the computation of the core matrix $\chi_\theta$ in (\ref{chitheta}).

\section{Computing the core matrix in state space}
\label{sec:statespace}
The state-space computation of the core matrix $\chi_\theta$ from (\ref{chitheta}) outlined below employs the following spectral factorizations
\begin{equation}
\label{cFcG}
    \phi(2i\theta \Psi(\lambda))
    =
    \cF(i\lambda)^*  \cH_\theta\cF(i\lambda),
    \qquad
    \Delta_\theta(\lambda)^{-1}
    =
    \cG_\theta(i\lambda) \cG_\theta(i\lambda)^*,
    \qquad
    \lambda \in \mR,
\end{equation}
for two of the matrices in the representation (\ref{fMF}), (\ref{fM}) (a factorization of the matrix $\varpi_\theta(\lambda)$ will also be discussed in what follows),
with
$\phi$,  $\Psi$, $\Delta_\theta$  from (\ref{phi}),  (\ref{Phi0_Psi0}), (\ref{Del}), so that \begin{equation}
\label{SDelS}
    S\Delta_\theta^{-1} S
    =
    S \cG_\theta  (S \cG_\theta  )^*
\end{equation}
in view of $S$ in (\ref{Oroot}) being Hermitian.   Here, $\cH_\theta \in \mH_\infty$ is an infinite block diagonal complex Hermitian matrix, and
$\cF : \mC\to \mC^{\infty\x r}$ and $\cG_\theta: \mC\to \mC^{m\x m}$ are  auxiliary transfer functions, analytic in the right half-plane, with $\cG_\theta$ being invertible on the imaginary axis: $\det\cG_\theta(s)\ne 0$ for any $s \in i\mR$ in view of (\ref{unipos}).
Prior to specifying these factors for (\ref{cFcG}), Lemma~\ref{lem:inner} below provides an insight into the structure of $\cG_\theta$ irrespective of a particular form of $\cF $, $\cH_\theta$. To this end, we associate with them a transfer function $\cK_\theta: \mC\to \mC^{\infty \x m}$ by
\begin{equation}
\label{cK}
  \cK_\theta
  :=
  \begin{bmatrix}
    \sqrt{\theta}\cF  F S\\
    \cG_\theta^{-1}
  \end{bmatrix},
\end{equation}
which also involves the closed-loop system transfer function $F$ and the matrix $S$ from (\ref{Oroot}).

\begin{lemma}
\label{lem:inner}
The second factorization in (\ref{cFcG}) is equivalent to the weighted isometry property
\begin{equation}
\label{cKinner}
  \cK_\theta(i\lambda)^*
  \begin{bmatrix}
    \cH_\theta & 0\\
    0 & I_m
  \end{bmatrix}
  \cK_\theta(i\lambda)
  =
  I_m,
  \qquad
  \lambda \in \mR,
\end{equation}
for the transfer function $\cK_\theta$ in (\ref{cK}), associated with the first factorization in  (\ref{cFcG}).
\end{lemma}
\begin{proof}
By substituting the first factorization from (\ref{cFcG}) into (\ref{Sig}), it follows that
\begin{equation}
\label{SigcF}
      \Sigma_\theta
    =
    S F^*
     \cF ^* \cH_\theta \cF  F S
     =
     (\cF  F S)^*
     \cH_\theta
     \cF  F S.
\end{equation}
By a combination of (\ref{Del}) with  (\ref{SigcF}), (\ref{cK}),
the second factorization in (\ref{cFcG})  holds if and only if
$
    0
    =
    \Delta_\theta
    -
    (\cG_\theta\cG_\theta^*)^{-1}
    =
    I_m
    -\theta
    (\cF  F S)^*
    \cH_\theta
     \cF  F S
     -
    \cG_\theta^{-*}
    \cG_\theta^{-1}
    =
    I_m
    -
    \cK_\theta^*
      {\small\begin{bmatrix}
    \cH_\theta & 0\\
    0 & I_m
  \end{bmatrix}}
    \cK_\theta
$
(with $(\cdot)^{-*}:= ((\cdot)^{-1})^*$), which is equivalent to (\ref{cKinner}).
\end{proof}

For specifying the factors $\cF $, $\cH_\theta$ in (\ref{cFcG}), which will be completed in Theorem~\ref{th:Delta} using Lemmas~\ref{lem:EOE}--\ref{lem:EOEk},  we note that in the series
\begin{equation}
\label{phiseries}
    \phi(2i\theta \Psi(\lambda))
    =
    \sum_{k=0}^{+\infty}
    \phi_k
    (2i\theta
    \Psi(\lambda)
    )^k ,
\end{equation}
each positive power $\Psi(\lambda)^k$ of the matrix $\Psi(\lambda)$ from (\ref{Phi0_Psi0}) is a product of $k$ copies of the transfer matrix $G(i\lambda)$ from (\ref{F0_G}) alternating with $G(i\lambda)^*$ and constant matrices:
\begin{equation}
\label{Psik}
    \Psi^k
    =
    (\cC G \mho G^* \cC^\rT)^k
    =
    \cC
    G \mho G^*
    (\Pi G \mho G^*)^{k-1}
    \cC^\rT,
    \qquad
    k \> 1,
\end{equation}
where $\mho$, $\Pi$ are the matrices from (\ref{cABPR}), (\ref{CC}). While the number of such alternations in $\Psi^k$ is $k$, it is  infinite in (\ref{phiseries}).
In order to arrive at the first factorization in (\ref{cFcG}), the product (\ref{Psik}) can be rearranged by moving all the  factors $G$ to the right and all the factors $G^*$ to the left and modifying the constant factors.  Despite the noncommutativity of the matrices $G$, $G^*$, $\mho$, $\Pi$,   this rearrangement is possible due to the following lemma, which is an adaptation of the ``system transposition'' technique from \cite[Lemma~2]{VP_2022_JFI} and can also be obtained by using the second resolvent identity \cite{K_1980}.

\begin{lemma}
\label{lem:EOE}
Suppose $U \in \mR^{(n+\nu)\x (n+\nu)}$ is an  arbitrary matrix  such that the solution $V:= \bL_\cA(U)$ of the ALE
\begin{equation}
\label{VALE}
    \cA V + V\cA^\rT + U = 0,
\end{equation}
associated with the Hurwitz matrix $\cA$ of the internally stable closed-loop system (\ref{dX_cAB}), is nonsingular.
Then the transfer function $G$ from (\ref{F0_G}) satisfies
\begin{equation}
\label{GUG}
    G(i\lambda) U G(i\lambda)^* = V G(i\lambda)^*  V^{-1}U V^{-1} G(i\lambda) V,
    \qquad
    \lambda \in \mR.
\end{equation}
\hfill$\blacksquare$
\end{lemma}

In particular, the ALE (\ref{VALE}) takes the form of the PR condition in (\ref{cABPR})  if Lemma~\ref{lem:EOE} is applied to the matrix $U:= \mho$, thus resulting in
the CCR matrix $V:= \Theta = \bL_\cA(\mho)$ in (\ref{Theta12_XCCR}), which is nonsingular in view of (\ref{detTheta}). In this case, (\ref{GUG}) leads to the rearrangement
\begin{equation}
\label{EmhoE}
  G \mho G^*
  =
  \Theta G^*  \Theta^{-1}\mho \Theta^{-1} G\Theta
\end{equation}
of intermediate factors in
\begin{equation}
\label{Psi1}
    \Psi
    =
    \cC G \mho G^* \cC^\rT
    =
    \cC \Theta G^*  \Theta^{-1}\mho \Theta^{-1} G\Theta \cC^\rT.
\end{equation}
The fact that,
in (\ref{GUG}) and its particular case (\ref{EmhoE}), the factor $G$ is moved to the right  while $G^*$ is moved to the left,  is reminiscent of the Wick ordering \cite{W_1950} for mixed products of quantum mechanical annihilation operators and their adjoints ---   creation operators (see also \cite[pp. 209--210]{J_1997}).

Lemma~\ref{lem:EOEk} below uses (\ref{EmhoE}) in order to extend (\ref{Psi1}) to arbitrary positive integer powers of $\Psi$. Its formulation employs three sequences of matrices $\alpha_k, \beta_k, \gamma_k  \in \mR^{(n+\nu)\x (n+\nu)}$ which are computed recursively according to a map $(\alpha_k,\beta_k,\gamma_{k-1}) \mapsto (\alpha_{k+1},\beta_{k+1},\gamma_k)$ as
\begin{equation}
\label{alfbetgamnext}
    \alpha_{k+1}
     = \gamma_k \beta_k,
     \
    \beta_{k+1}
    =
    \gamma_k^{-1} \alpha_k \Pi^{\delta_{k1}}\gamma_{k-1}\gamma_k^{-1},
    \
    \gamma_k
    =
    \bL_\cA(\alpha_k \Pi^{\delta_{k1}}\gamma_{k-1}),
    \qquad
    k\> 1,
\end{equation}
with the initial conditions
\begin{equation}
\label{alfbetgam0}
        \alpha_1
    =
    \gamma_0
    =
    \bL_\cA(\mho)
    =
    \Theta,
    \qquad
    \beta_1 = \Theta^{-1}\mho \Theta^{-1}.
\end{equation}
Here,
$
    \Pi^{\delta_{k1}}
    =
    \left\{
    {\small\begin{matrix}
      \Pi & {\rm if}\ k=1\\
      I_{n+\nu} & {\rm otherwise}
    \end{matrix}}
    \right.
$, where $\delta_{jk}$  is the Kronecket delta,   and $\Pi$ is the matrix from (\ref{CC}).
For convenience, the first equality in (\ref{alfbetgamnext}) is extended to $k=0$ as $\alpha_1 = \gamma_0 \beta_0$ by letting
\begin{equation}
\label{bet0}
    \beta_0:= \gamma_0^{-1} \alpha_1 = I_{n+\nu},
\end{equation}
 in accordance with the first equality in (\ref{alfbetgam0}). In order for the recurrence equations (\ref{alfbetgamnext}) to be valid for all $k$, it is assumed that 
  \begin{equation}
 \label{gamdet}
   \det \gamma_k \ne 0,
   \qquad
   k\> 1.
 \end{equation}
In particular, by letting $k:=1$ in the third of the equalities  (\ref{alfbetgamnext}) and using (\ref{alfbetgam0}), it follows that
\begin{equation}
\label{gam1}
    \gamma_1
    =
    \bL_\cA(\alpha_1 \Pi\gamma_0)
    =
    \bL_\cA(\Theta \Pi \Theta).
\end{equation}
Therefore, since $\Theta \Pi \Theta = -\Theta \cC^\rT (\Theta\cC^\rT)^\rT$ due to (\ref{CC}) and the antisymmetry of $\Theta$, the condition $\det \gamma_1\ne 0$ for the matrix (\ref{gam1}) is equivalent to the controllability of the pair $(\cA,\Theta \cC^\rT)$.
The following lemma is similar to \cite[Lemma~3]{VP_2022_JFI} and provides relevant properties of the above matrices. 

\begin{lemma}
\label{lem:alfbetgam}
The matrices $\beta_k$, $\gamma_k$, defined by (\ref{alfbetgamnext})--(\ref{bet0}) subject to (\ref{gamdet}),   have the opposite symmetric properties:
\begin{equation}
\label{betgam+-}
    \beta_k^\rT = (-1)^k \beta_k,
    \qquad
    \gamma_k^\rT = -(-1)^k \gamma_k,
    \qquad
    k \> 0.
\end{equation}
Furthermore, the matrices $\beta_k$ with even $k$ have an alternating definiteness in the sense that 
    $(-1)^\ell\beta_{2\ell} \succ 0$ for any $
    \ell \> 0$.
\hfill$\blacksquare$
\end{lemma}

The following lemma clarifies the role of the sequences of matrices $\alpha_k$, $\beta_k$, $\gamma_k$ in a Wick-like ordering which will be used in the proof of spectral factorizations in Lemma~\ref{lem:EOEk}.

\begin{lemma}
\label{lem:pull}
The matrices $\alpha_k$, $\beta_k$, $\gamma_k$, defined by (\ref{alfbetgamnext})--(\ref{bet0}) subject to (\ref{gamdet}), and the function $G$ in (\ref{F0_G}) satisfy
\begin{equation}
\label{pull}
      \Big(
    \lprod_{j=1}^k
    G(i\lambda) \alpha_j
    \Big)
    \Pi
    G(i\lambda)\mho G(i\lambda)^*
    =
    \gamma_k
    G(i\lambda)^*
    \beta_{k+1}
    \lprod_{j=1}^{k+1}
    G (i\lambda)\alpha_j,
    \quad
    \lambda \in \mR,
    \
    k \> 1,
\end{equation}
where $\lprod(\cdot)$ is the leftwards  ordered product, and $\mho$, $\Pi$ are the matrices from (\ref{cABPR}), (\ref{CC}).
\end{lemma}
\begin{proof}
We will prove the relation (\ref{pull}) by induction over $k\> 1$. Its validity for $k=1$  is verified  by
\begin{align*}
    G\alpha_1\Pi
    G\mho G^*
    & =
    G
    \alpha_1 \Pi\Theta
    G^*
    \overbrace{\Theta^{-1}\mho \Theta^{-1} }^{\beta_1}
    G\Theta
  =
  G\alpha_1\Pi
  \Theta G^*  \beta_1 G\Theta\\
  & =
  \gamma_1
  G^*
  \underbrace{\gamma_1^{-1}\alpha_1\Pi  \Theta \gamma_1^{-1}}_{\beta_2}
  G
  \underbrace{\gamma_1\beta_1}_{\alpha_2}
  G\Theta
  =
  \gamma_1
  G^*
  \beta_2
  G
  \alpha_2 G\alpha_1,
\end{align*}
where the first equality uses (\ref{EmhoE}), the second equality applies (\ref{GUG}) of  Lemma~\ref{lem:EOE} to $G\alpha_1 \Pi
  \Theta G^*$, and use is also made of the matrices $\alpha_2$, $\beta_2$,  $\gamma_1$ from (\ref{alfbetgamnext}) along with (\ref{alfbetgam0}). Now, if (\ref{pull}) is already proved for some $k\> 1$, then
  \begin{align*}
  \nonumber
      \Big(
    \lprod_{j=1}^{k+1}&
    G\alpha_j
    \Big)
    \Pi
    G\mho G^*
    =
    G\alpha_{k+1}
      \Big(
    \lprod_{j=1}^k
    G\alpha_j
    \Big)
    \Pi
    G\mho G^*
    =
    G \alpha_{k+1}
    \gamma_k
    G^*
    \beta_{k+1}
    \lprod_{j=1}^{k+1}
    G \alpha_j \\
    & =
    \gamma_{k+1}
    G^*
    \underbrace{\gamma_{k+1}^{-1}\alpha_{k+1}\gamma_k \gamma_{k+1}^{-1}}_{\beta_{k+2}}
    G
    \underbrace{\gamma_{k+1}\beta_{k+1}}_{\alpha_{k+2}}
    \lprod_{j=1}^{k+1}
    G \alpha_j
    =
    \gamma_{k+1}
    G^*
    \beta_{k+2}
    \lprod_{j=1}^{k+2}
    G \alpha_j,
\end{align*}
where Lemma~\ref{lem:EOE} and (\ref{alfbetgamnext}) are applied again, proving (\ref{pull}) for the next value $k+1$ and completing the induction step.
\end{proof}

The relation (\ref{pull}) (and its inductive proof) can be interpreted as ``pulling'' the factor $G^*$ leftwards through the product of the $G$ factors (and constant matrices between them) until $G^*$ is to the left of all the $G$ factors. This procedure is used in the proof of the  following lemma, which is  similar to \cite[Theorem 1]{VP_2022_JFI} and employs a sequence of transfer  functions
\begin{equation}
\label{GkC}
    G_k(s)
    :=
    \left\{
    {\begin{matrix}
    I_r &   {\rm if}\  k=0\\
    \Big(
        \lprod_{j=1}^k
    G(s)\alpha_j
    \Big)
    \cC^\rT
& {\rm if}\  k \> 1
    \end{matrix}}
    \right.,
    \qquad
    s \in \mC,
\end{equation}
defined
in terms of $G$ from (\ref{F0_G}) and the matrices $\alpha_k$ from (\ref{alfbetgamnext}),  (\ref{alfbetgam0}) and taking values in $\mC^{(n+\nu)\x r}$ for $k\> 1$.

\begin{lemma}
\label{lem:EOEk}
The powers of the function $\Psi$ in (\ref{Phi0_Psi0}) are factorized in terms of (\ref{GkC}) as
\begin{equation}
\label{GOGk}
    \Psi(\lambda)^k
    =
    (-1)^k
    G_k(i\lambda)^*
    \beta _k
    G_k(i\lambda),
    \qquad
    \lambda \in \mR,
    \
    k \> 0,
\end{equation}
where $\beta_k$ are the matrices given by (\ref{alfbetgamnext})--(\ref{bet0}) subject to the condition (\ref{gamdet}). 
\end{lemma}

\begin{proof}
While the relation (\ref{GOGk}) holds trivially for $k=0$ due to (\ref{bet0}), (\ref{GkC}), we will prove it by induction on $k\> 1$.
For $k=1$, its validity  follows from (\ref{Psi1}), (\ref{alfbetgam0}) as
$
    \Psi
  =
  \cC \Theta G^*  \Theta^{-1}\mho \Theta^{-1} G\Theta \cC^\rT
  =
  -\underbrace{\cC  \alpha_1^\rT G^*}_{G_1^*} \beta_1 \underbrace{G \alpha_1 \cC^\rT}_{G_1}
  =
  -G_1^* \beta_1 G_1
$
in view of (\ref{GkC}) and the antisymmetry of the CCR matrix $\Theta$ in (\ref{Theta12_XCCR}). Now, suppose (\ref{GOGk}) is already proved for some $k\> 1$.
Then the next power of the matrix $\Psi(\lambda)$ takes the form
\begin{align*}
    \Psi^{k+1}
    & =
    \Psi^k \Psi
    =
    (-1)^k G_k^* \beta_k G_k \Psi
    =
    (-1)^k G_k^* \beta_k
    \Big(
    \lprod_{j=1}^k
    G \alpha_j
    \Big)
    \Pi
    G \mho G^* \cC^\rT\\
    & =
    (-1)^k
    G_k^*
    \underbrace{\beta_k \gamma_k}_{-\alpha_{k+1}^\rT}
    G^*
    \beta_{k+1}
    \Big(
    \lprod_{j=1}^{k+1}
    G \alpha_j
    \Big)
    \cC^\rT
    =
    (-1)^{k+1}
    G_{k+1}^*
    \beta_{k+1}
    G_{k+1},
\end{align*}
where (\ref{pull}) of Lemma~\ref{lem:pull} is applied along with (\ref{GkC}), (\ref{alfbetgamnext}) and the symmetric properties (\ref{betgam+-}), thereby completing the induction step.
\end{proof}

The following theorem is a corollary of Lemma~\ref{lem:EOEk} and, similarly to \cite[Theorem~2]{VP_2022_JFI},   establishes the first factorization in (\ref{cFcG}) with the transfer function
\begin{equation}
\label{cF}
    \cF
    :=
    \begin{bmatrix}
       G_0\\
      G_1\\
      G_2\\
      G_3\\
      \vdots
    \end{bmatrix}
    =
    \begin{bmatrix}
      I_r\\
      G \alpha_1 \cC^\rT\\
      G \alpha_2 G \alpha_1 \cC^\rT\\
      G \alpha_3 G \alpha_2 G \alpha_1 \cC^\rT\\
      \vdots
    \end{bmatrix}
    =
    \left[{\small\begin{array}{cccc|c}
  \cA   & 0 & 0 &\ldots & \alpha_1 \cC^\rT\\
  \alpha_2  & \cA & 0 & \ldots  & 0\\
  0   & \alpha_3 & \cA & \ldots  & 0\\
  \vdots    &\vdots   & \vdots  & \ddots  &\vdots  \\
  \hline
  0 & 0 & 0 & \ldots  & I_r\\
  I_{n+\nu} & 0 & 0 & \ldots  & 0\\
  0 & I_{n+\nu} & 0 & \ldots  & 0\\
  0 & 0 & I_{n+\nu} & \ldots  & 0\\
  \vdots & \vdots & \vdots & \ddots  & \vdots
\end{array}}\right]
\end{equation}
which uses (\ref{GkC})
(with the horizontal and vertical lines separating the state-space realization matrices),  and the matrix
\begin{equation}
\label{cH}
    \cH_\theta
    :=
    \diag(I_r,
        \diag_{k\> 1}
        ((-2i\theta)^k
        \phi_k \beta_k)),
\end{equation}
defined in terms of the coefficients $\phi_k$ from  (\ref{phi}) and the matrices $\beta_k$ from (\ref{alfbetgamnext})--(\ref{bet0}).  Note that $\cH_\theta$ in (\ref{cH})  is Hermitian due to the first equality in (\ref{betgam+-}) of Lemma~\ref{lem:alfbetgam}.

 \begin{theorem}
\label{th:Delta}
Suppose the condition (\ref{gamdet}) is satisfied. Then the transfer function $\cF$ in   (\ref{cF}) and the matrix $\cH_\theta$ in (\ref{cH}) deliver the first factorization in (\ref{cFcG}).
\end{theorem}
\begin{proof}
By substituting the factorizations  (\ref{GOGk}) into the series (\ref{phiseries}) and using $\phi_0 = 1$ along with (\ref{cF}), (\ref{cH}), it follows that
$
    \phi(2i\theta \Psi(\lambda))
     =
    \sum_{k=0}^{+\infty}
    \phi_k
    (2i\theta
    \Psi(\lambda)
    )^k
    =
    \sum_{k=0}^{+\infty}
    \phi_k
    (-2i\theta)^k
    G_k(i\lambda)^*
    \beta _k
    G_k(i\lambda)
    =
    \cF(i\lambda)^*\cH_\theta \cF(i\lambda)
$
for any $\lambda \in \mR$,
thus proving the first equality in (\ref{cFcG}).
\end{proof}

The transfer function $\cF$ in (\ref{cF}),   which provides the spectral factorisation in (\ref{cFcG}),  is implemented as an infinite cascade of classical linear time invariant (LTI)  systems shown in Fig.~\ref{fig:cF}.
\begin{figure}[htbp]
\centering
\unitlength=0.5mm
\linethickness{0.6pt}
\begin{picture}(170.00,37.00)
    \put(5,13){\dashbox(165,20)[cc]{}}
    \put(7,31){\makebox(0,0)[lt]{{\small$\cF$}}}
    \put(150,18){\framebox(10,10)[cc]{\small$\cC^\rT$}}
    \put(130,18){\framebox(10,10)[cc]{\small$\alpha_1$}}
    \put(110,18){\framebox(10,10)[cc]{\small$G$}}
    \put(90,18){\framebox(10,10)[cc]{\small$\alpha_2$}}
    \put(70,18){\framebox(10,10)[cc]{\small$G$}}
    \put(50,18){\framebox(10,10)[cc]{\small$\alpha_3$}}
    \put(30,18){\framebox(10,10)[cc]{\small$G$}}

    \put(183,23){\makebox(0,0)[cc]{{\small$v$}}}

    \put(180,23){\vector(-1,0){20}}
    \put(150,23){\vector(-1,0){10}}
    \put(130,23){\vector(-1,0){10}}
    \put(110,23){\vector(-1,0){10}}
    \put(90,23){\vector(-1,0){10}}
    \put(70,23){\vector(-1,0){10}}
    \put(50,23){\vector(-1,0){10}}
    \put(30,23){\vector(-1,0){10}}
    \put(15,23){\makebox(0,0)[cc]{{$\cdots$}}}
    \put(15,3){\makebox(0,0)[cc]{{$\cdots$}}}

    \put(165,0){\makebox(0,0)[cc]{{\small$v$}}}
    \put(105,0){\makebox(0,0)[cc]{{\small$z_1$}}}
    \put(65,0){\makebox(0,0)[cc]{{\small$z_2$}}}
    \put(25,0){\makebox(0,0)[cc]{{\small$z_3$}}}
    \put(165,23){\vector(0,-1){20}}
    \put(105,23){\vector(0,-1){20}}
    \put(65,23){\vector(0,-1){20}}
    \put(25,23){\vector(0,-1){20}}
\end{picture}
\caption{An infinite cascade of LTI systems with the common transfer function $G$ from (\ref{F0_G}) (and the matrices $\cC^\rT$,   $\alpha_k$ from (\ref{Z_U_cC}),   (\ref{alfbetgamnext}), (\ref{alfbetgam0}) as intermediate static factors) which form a system with the transfer  function $\cF$ in  (\ref{cF}).
}
\label{fig:cF}
\end{figure}
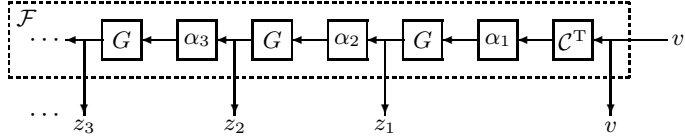
This cascaded system has an $\mR^r$-valued input $v$ and an $\mR^r \x (\mR^{(n+\nu)})^\infty$-valued output $(v,z_1, z_2, z_3, \ldots)$, where $z_k$ is the $\mR^{n+\nu}$-valued internal state of the $k$th subsystem $G$ in this cascade.
A similar infinite cascade state-space realization holds for the transfer function
\begin{equation}
\label{cFFS}
    \cF FS
    :=
    \begin{bmatrix}
        I_r\\
      G_1\\
      G_2\\
      \vdots
    \end{bmatrix}
    FS
    =
    \begin{bmatrix}
      \cC G \\
      G \alpha_1 \Pi G \\
      G \alpha_2 G \alpha_1 \Pi G \\
      \vdots
    \end{bmatrix}
    \cB S
    =
        \left[
    \begin{array}{c|c}
        \sA& \sB \\
        \hline
        \sC & 0
    \end{array}
    \right],
\end{equation}
which is obtained by combining (\ref{F0_G}) with (\ref{cF}), (\ref{CC})  and has
the state-space matrices
\begin{equation}
\label{sABC}
  \sA
  :=
  \begin{bmatrix}
  \cA   & 0 & 0 &\ldots\\
  \alpha_1\Pi  & \cA & 0 & \ldots\\
  0   & \alpha_2 & \cA & \ldots\\
  \vdots    &\vdots   & \vdots  & \ddots \\
  \end{bmatrix},
  \quad
  \sB
  :=
  \begin{bmatrix}
    \cB S\\
    0\\
    0\\
    \vdots
  \end{bmatrix},
  \quad
  \sC
  :=
  \begin{bmatrix}
  \cC & 0 & 0 & \ldots\\
  0 & I_{n+\nu} & 0 & \ldots\\
  0 & 0 & I_{n+\nu} & \ldots\\
  \vdots & \vdots & \vdots & \ddots
  \end{bmatrix},
\end{equation}
where $\sB \in \mC^{\infty\x m}$.
Now, consider the following algebraic Riccati equation (ARE) with respect to a matrix $\sQ_\theta \in \mH_\infty$:
\begin{equation}
\label{ARE}
    \sA^\rT \sQ_\theta + \sQ_\theta\sA
    + \theta \sC^\rT \cH_\theta \sC
    +
    \sL_\theta^* \sL_\theta = 0,
    \qquad
    \sL_\theta
    :=
    \sB^* \sQ_\theta,
\end{equation}
where $\cH_\theta$ is given by (\ref{cH}), and  $\sL_\theta^* \sL_\theta = \sQ_\theta \sB \sB^* \sQ_\theta$, with $\sB \sB^* = \diag(\cB \Omega^\rT \cB^\rT, 0) \in \mH_\infty^+$ in view of the structure of the matrix $\sB$ in (\ref{sABC}) and by (\ref{Oroot}). Its solution $\sQ_\theta$ is said to be
\emph{stabilizing} if the spectrum of the infinite matrix $\sA + \sB \sL_\theta$ is in the left half-plane, thus giving rise to a $\mC^{m\x m}$-valued transfer function
\begin{equation}
\label{cG}
    \cG_\theta
    :=
    \left[
    \begin{array}{c|c}
        \sA +\sB \sL_\theta & \sB \\
        \hline
        \sL_\theta & I_m
    \end{array}
    \right],
\end{equation}
analytic in the right-half plane. The state-space realization (\ref{cG}) is shown in Fig.~\ref{fig:cG}.   
\begin{figure}[htbp]
\unitlength=0.7mm
\linethickness{0.6pt}
\centering
\begin{picture}(80,35.00)
    \put(25,0){\dashbox(40,35)[cc]{}}
    \put(27,33){\makebox(0,0)[lt]{{\small$\cG_\theta$}}}

    \put(30,5){\framebox(10,10)[lc]{\scriptsize$\, \cF\! F\! S$}}
    \put(55,25){\circle{6}}
    \put(55,25){\makebox(0,0)[cc]{\small$+$}}
    \put(50,5){\framebox(10,10)[cc]{\scriptsize$\sL_\theta$}}
    \put(55,15){\vector(0,1){7}}
    \put(39,10){\makebox(0,0)[cc]{\tiny$\bullet$}}
    \put(35,25){\vector(0,-1){10}}
    \put(39,10){\vector(1,0){11}}
    \put(72,25){\vector(-1,0){14}}
    \put(52,25){\vector(-1,0){35}}

    \put(13,25){\makebox(0,0)[cc]{\small$z$}}
    \put(75,25){\makebox(0,0)[cc]{\small$v$}}
\end{picture}
\caption{An implementation of the transfer function $\cG_\theta$ from (\ref{cG}) with an input $v$ and output $z$ (both  $\mC^m$-valued). The internal state of the system $\cF F S$ from (\ref{cFFS}) is weighted by a static matrix $\sL_\theta \in \mC^{m\x \infty}$ and enters the system  in a mixture with the input $v$.
}
\label{fig:cG}
\end{figure}
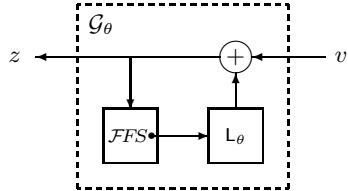

\begin{theorem}
\label{th:cG}
The transfer function $\cG_\theta$ in (\ref{cG}), specified by (\ref{sABC}) and the stabilizing  solution $\sQ_\theta$ of the ARE (\ref{ARE}),   satisfies the second factorization in (\ref{cFcG}).
\end{theorem}
\begin{proof}
From (\ref{cG}) (see also Fig.~\ref{fig:cG}), it follows that the inverse transfer matrix $\cG_\theta^{-1}$ has the state-space realization
\begin{equation}
\label{cGinv}
    \cG_\theta^{-1}
    =
    \left[
    \begin{array}{c|c}
        \sA & \sB \\
        \hline
        -\sL_\theta & I_m
    \end{array}
    \right].
\end{equation}
Since (\ref{cFFS}) and (\ref{cGinv}) share the input-state dynamics, the transfer function (\ref{cK}) is realized as
\begin{equation}
\label{cK1}
    \cK_\theta
    =
    \left[
    \begin{array}{c|c}
      \sA & \sB\\
      \hline
      \sqrt{\theta} \sC & 0\\
      -\sL_\theta & I_m
    \end{array}
    \right],
\end{equation}
with the solution  $\sQ_\theta$ of the ARE  (\ref{ARE}) being a weighted observability Gramian for  (\ref{cK1}).
For any $\lambda \in \mR$, the transfer matrix $\cK_\theta(i\lambda) = {\small\begin{bmatrix}
  \sqrt{\theta}\sC \\
  - \sL_\theta
\end{bmatrix}}(i\lambda I_\infty - \sA)^{-1}\sB + {\small\begin{bmatrix}
  0\\
  I_m
\end{bmatrix}}$ satisfies
\begin{align*}
    \cK_\theta(i\lambda)^*&
    \begin{bmatrix}
      \cH_\theta & 0\\
      0 & I_m
    \end{bmatrix}
    \cK_\theta(i\lambda)
    =
    -\sB^*(i\lambda I_\infty + \sA^\rT)^{-1}
    (\theta \sC^\rT \cH_\theta \sC
    +
    \sL_\theta^* \sL_\theta)
    (i\lambda I_\infty - \sA)^{-1}\sB\\
     &
    +
    \sB^*(i\lambda I_\infty + \sA^\rT)^{-1}
    \sL_\theta^*
    -\sL_\theta
    (i\lambda I_\infty - \sA)^{-1}\sB
    +I_m\\
    = &
    \sB^*(i\lambda I_\infty + \sA^\rT)^{-1}
    (\sA^\rT \sQ_\theta + \sQ_\theta\sA)
    (i\lambda I_\infty - \sA)^{-1}\sB
    +
    \sB^*(i\lambda I_\infty + \sA^\rT)^{-1}
    \sL_\theta^*
    -\sL_\theta
    (i\lambda I_\infty - \sA)^{-1}\sB
    +I_m    \\
     = &
    \sB^*(i\lambda I_\infty + \sA^\rT)^{-1}
    (\sL_\theta^*-\sQ_\theta \sB)
    +
    (\sB^*\sQ_\theta-\sL_\theta)
    (i\lambda I_\infty - \sA)^{-1}\sB
    +I_m
    =
    I_m
\end{align*}
in view of (\ref{ARE}),   thus leading to (\ref{cKinner}) (here, $I_\infty$ is the infinite identity matrix). Therefore, by Lemma~\ref{lem:inner}, the transfer function $\cG_\theta$ in (\ref{cG}) delivers the second factorization in (\ref{cFcG}).
\end{proof}

Note that the state-space realization of the factor $S \cG_\theta$ in (\ref{SDelS}) can be  obtained from (\ref{cG}) as
\begin{equation}
\label{ScG}
    S\cG_\theta
    =
    \left[
    \begin{array}{c|c}
        \sA +\sB \sL_\theta & \sB \\
        \hline
        S \sL_\theta & S
    \end{array}
    \right].
\end{equation}
In order to complete the state-space representation of the operator $\fM_{\theta,F}$ in  (\ref{fM}), consider its third matrix $\varpi_\theta(\lambda)$  in (\ref{varpitheta}). A combination of the second equality from (\ref{dfdalpha}) with the factorizations  (\ref{GOGk}), (\ref{SDelS}) yields
\begin{align}
\nonumber
    \varpi_\theta
     & =
     \sum_{j,k=0}^{+\infty}
     \phi_{j+k+1}
     (2i\theta)^{j+k}
     \Psi^j
     FS\Delta_\theta^{-1} SF^*
     \Psi^k
     =
     \sum_{j,k=0}^{+\infty}
     \phi_{j+k+1}
     (-2i\theta)^{j+k}
     G_j^*\beta_j G_j
     FS\Delta_\theta^{-1} SF^*
     G_k^*\beta_k G_k\\
\label{varpiseries}
     & =
     \sum_{j,k=0}^{+\infty}
     \phi_{j+k+1}
     (-2i\theta)^{j+k}
     G_j^*\beta_j G_j
     FS \cG_\theta  (G_k F S \cG_\theta  )^* \beta_k G_k.
\end{align}
Therefore, by substituting (\ref{cFcG}), (\ref{SDelS}), (\ref{varpiseries}) into (\ref{fM}), (\ref{fMF}), the core matrix in (\ref{chitheta}) takes the form
\begin{align}
\nonumber
    \chi_\theta
    =&
  \frac{1}{2\pi}
  \int_{\mR}
  \Re
  \fP
  \Big(
    \begin{bmatrix}
      G^* \cC^\rT \\
      I_r
    \end{bmatrix}
    \Big(\cF^* \cH_\theta \cF F S\cG_\theta (S\cG_\theta)^*\\
\label{chitheta1}
     & -
          \sum_{j,k=0}^{+\infty}
     \phi_{j+k+1}
     (-2i\theta)^{j+k+1}
     G_j^*\beta_j G_j
     FS \cG_\theta  (G_k F S \cG_\theta  )^* \beta_k G_k F J
     \Big)
    \begin{bmatrix}
      \cB^\rT G^* &I_m
    \end{bmatrix}
      \Big)
      \rd \lambda.
\end{align}
Despite the infinite number of  terms in (\ref{chitheta1}), each of them beyond the zero bottom-right $(r\x m)$-block of the core matrix (in view of the projection operator $\fP$)   is the real part of a mixed moment
\begin{equation}
\label{intF15}
    \bM(F_1, \ldots, F_5)
    :=
    \frac{1}{2\pi}
    \int_{\mR}
    F_1^* F_2 F_3^*F_4 F_5^*
    \rd \lambda
\end{equation}
of proper stable transfer functions $F_1, \ldots, F_5$ such that at least two of them are strictly proper, thus making the above integral convergent. Omitting the trailing identity matrices in (\ref{intF15}) for brevity, so that, for example,
\begin{equation}
\label{intF123}
    \bM(F_1, F_2, F_3)
    :=
    \frac{1}{2\pi}
    \int_{\mR}
    F_1(i\lambda)^* F_2(i\lambda) F_3(i\lambda)^*
    \rd \lambda
    =
    \bM(F_1, F_2, F_3, I, I)
\end{equation}
(where $I$ is an appropriately dimensioned identity matrix),
we note that the $(1,1)$-block of (\ref{chitheta1}) is represented as
\begin{align}
\nonumber
    (\chi_\theta)_{11}
    =&
  \Re
  \Big(
    \bM(\cF \cC G, \cH_\theta \cF F S\cG_\theta , G\cB S \cG_\theta)\\
\label{chitheta11}
     & -
          \sum_{j,k=0}^{+\infty}
     \phi_{j+k+1}
     (-2i\theta)^{j+k+1}
     \bM(G_j \cC G, \beta_j G_jFS \cG_\theta, G_k F S \cG_\theta, \beta_k G_k F J, G\cB)
      \Big).
\end{align}
The other nontrivial blocks $(\chi_\theta)_{12}$, $(\chi_\theta)_{21}$ of the core matrix admit similar representations.
Since the mixed products of transfer functions in (\ref{intF15}) have at most four alternations of stable and antistable factors, these integrals can be expressed in terms of Gramians as demonstrated in Appendix~\ref{sec:mixint}.
Similarly to \cite[Section 7]{VP_2022_JFI}, the infinite-dimensional systems $\cF$, $\cG_\theta$, $S\cG_\theta$ in (\ref{cF}), (\ref{cG}), (\ref{ScG})  can be used in practice (including the numerical solution of the ARE (\ref{ARE})) in the form of their truncations up to moderate orders due to the factorially fast decay of the coefficients $\phi_k$ in (\ref{phi}) which are present in the matrix $\cH_\theta$ in (\ref{cH}) and the series in (\ref{chitheta1}).

\section{Conclusion}
\label{sec:conc}

We have considered a risk-sensitive optimal control problem for a measurement-free coherent quantum feedback interconnection, where both the plant and the controller are OQHOs with field-mediated  coupling. In this setting,  the controller has to stabilize the closed-loop system and minimize the infinite-horizon growth rate of a quadratic-exponential penalty on the plant variables and the controller output variables in the invariant Gaussian state when the system is subject to vacuum quantum noises. In contrast to classical risk-sensitive control, the  QEF rate (as a cost functional in the quantum control problem) depends not only on the statistical properties of the criterion process but also on its two-point commutator kernel,  which is also affected by the choice of a coherent quantum controller  parameterized by the energy and coupling matrices   due to the PR constraints. We have obtained first-order necessary conditions of optimality in this class of controllers by computing the partial Frechet derivatives of the cost functional in frequency domain with respect to the controller parameters. This computation has been reduced to that of the core matrix, through which a particular form of the cost enters the derivatives.  The core matrix has been shown to play a key role in infinitesimal equivalence of the coherent quantum risk-sensitive and weighted CQLQG control problems. We have also outlined a state-space calculation of the core matrix using spectral factorizations in terms of infinite cascades of auxiliary classical systems. In truncated form, they are applicable to a gradient descent algorithm for numerical controller synthesis initialized with an optimal CQLQG  controller.

%

\appendix

\section{Differential identities for functions of matrices}\label{sec:matfun}
The computation of Frechet derivatives in Section~\ref{sec:opt} employs the following lemma,
which is similar to \cite[Theorem 4.11]{NH_1995} (see also \cite[Eq. (3.13)]{H_2008} and \cite{S_1987}) 
and is provided here for completeness.

\begin{lemma}
\label{lem:fmat'}
Suppose $f$ is an analytic  function on a domain (of the complex plane $\mC$) containing the spectrum of a matrix $\alpha \in \mC^{r\x r}$. Then
the Gateaux derivative $\d_\alpha f(\alpha)(\beta):= \lim_{\eps \to 0}(\frac{1}{\eps}(f(\alpha + \eps \beta)-f(\alpha)))$ in the direction $\beta \in \mC^{r\x r}$  can be computed as
\begin{equation}
\label{fmat'}
    \d_\alpha
    f(\alpha)(\beta)
     =
    \left(
    f
    (\sigma)
    \right)_{21},
\end{equation}
with $(\cdot)_{21}$ denoting the bottom left $( r\x  r)$-block  of the matrix $f(\sigma) \in \mC^{2r\x 2r}$, where
\begin{equation}
\label{sig}
        \sigma
        :=
        \begin{bmatrix}
      \alpha & 0\\
      \beta & \alpha
    \end{bmatrix}.
\end{equation}
The corresponding first variation $\delta f(\alpha) = \d_\alpha f(\alpha)(\delta\alpha)$  of $f(\alpha)$ with respect to $\alpha$ satisfies
\begin{equation}
\label{fab0}
  \Tr (\tau \delta f(\alpha))
  =
  \Tr
  (\d_\alpha f(\alpha)(\tau) \delta \alpha),
  \qquad
  \alpha, \delta\alpha, \tau \in \mC^{r\x r}.
\end{equation}
\end{lemma}
\begin{proof}
The conditions of the lemma secure the Cauchy integral representations for the matrices
\begin{align}
\label{fres_ffres}
    f(\alpha)
    & =
    \frac{1}{2\pi i}
    \oint_C
    f(z)
    (zI_r - \alpha)^{-1}
    \rd z,
    \qquad
    f(\sigma)
     =
    \frac{1}{2\pi i}
    \oint_C
    f(z)
    (zI_{2r}-\sigma)^{-1}
   \rd z,
\end{align}
where the integration is over a counterclockwise oriented closed contour $C$, which lies in the analyticity domain of $f$ and embraces the spectrum of $\alpha$ (and hence,  that of the matrix $\sigma$ in (\ref{sig})).
Since
$
    \d_\alpha ((zI_r - \alpha)^{-1})(\beta)
    =
    (zI_r - \alpha)^{-1}
    \beta
    (zI_r - \alpha)^{-1}
$
is the $(2,1)$-block of the matrix
$
    (zI_{2r}- \sigma)^{-1}
    =
    {\scriptsize\begin{bmatrix}
      (zI_r - \alpha)^{-1} & 0\\
      (zI_r - \alpha)^{-1}\beta (zI_r - \alpha)^{-1} & (zI_r - \alpha)^{-1}
    \end{bmatrix}}
$,
then, in view of the first equality in (\ref{fres_ffres}),
\begin{equation}
\label{int}
    \d_\alpha f(\alpha)(\beta)
     =
        \frac{1}{2\pi i}
    \oint_C
    f(z)
    \d_\alpha
    ((zI_r - \alpha)^{-1})(\beta)
    \rd z
    =
        \frac{1}{2\pi i}
    \oint_C
    f(z)
    (zI_r - \alpha)^{-1}\beta (zI_r - \alpha)^{-1}
    \rd z
\end{equation}
coincides with the corresponding block of the matrix $f(\sigma)$ from (\ref{fres_ffres}), thus establishing (\ref{fmat'}). The relation (\ref{int}) allows the Frechet derivative $\d_\alpha f(\alpha)$, which is a linear operator on $\mC^{r\x r}$,  to be represented as an integral over a parametric family of sandwich operators (\ref{sand}):
\begin{equation}
\label{int1}
    \d_\alpha f(\alpha)
     =
        \frac{1}{2\pi i}
    \oint_C
    f(z)
    \[[[
        (zI_r - \alpha)^{-1},
        (zI_r - \alpha)^{-1}
    \]]]
    \rd z.
\end{equation}
Therefore,  the identity (\ref{fab0}) for the first variation $\delta f(\alpha) = \d_\alpha f(\alpha)(\delta \alpha)$ is inherited from sandwich operators of the form $\[[[\varphi, \varphi\]]]$ with arbitrary $\varphi \in \mC^{r\x r}$,  since for any such operator,
$
    \Tr(\tau \[[[\varphi, \varphi\]]](\beta))
    =
    \Tr(\tau \varphi \beta  \varphi)
    =
    \Tr(\varphi\tau \varphi \beta  )
    =
    \Tr (\[[[\varphi, \varphi\]]](\tau) \beta)
$
by the cyclic property of the matrix trace.
\end{proof}

Note that if $f$ is an entire function,  with a globally convergent power series expansion
\begin{equation}
\label{fexp}
    f(z)
    =
    \sum_{k=0}^{+\infty} f_k z^k,
    \qquad
    z \in \mC
\end{equation}
(where $f_k$ are complex coefficients), then the Frechet derivative of $f(\alpha)$ with respect to $\alpha \in \mC^{r\x r}$ admits an alternative (compared to (\ref{int1})) expansion over the sandwich operators (\ref{sand}):
\begin{equation}
\label{dfdalpha}
    \d_\alpha f(\alpha)
    =
    \sum_{k=1}^{+\infty}
    f_k
    \sum_{j=0}^{k-1}
    \[[[
        \alpha^j,
        \alpha^{k-1-j}
    \]]]
    =
    \sum_{j,k=0}^{+\infty}
    f_{j+k+1}
    \[[[
        \alpha^j,
        \alpha^k
    \]]].
\end{equation}
Since the operator norm of (\ref{sand}), induced by the Frobenius norm $\|\cdot\|_\rF$,  is expressed as
\begin{equation}
\label{sbound}
    \|\[[[\sigma, \tau\]]]\|
    =
    \|\sigma\|\|\tau\|
\end{equation}
in terms of the operator norms of matrices,
then by the submultiplicativity of operator norms, the first equality in (\ref{dfdalpha}) implies that
\begin{equation}
\label{dfdanorm}
    \|\d_\alpha f(\alpha)\|
    \<
    \sum_{k=1}^{+\infty}
    |f_k|
    \sum_{j=0}^{k-1}
    \|
        \alpha^j\|
        \|\alpha^{k-1-j}\|
        \<
    \sum_{k=1}^{+\infty}
    k
    |f_k|
    \|\alpha\|^{k-1}.
\end{equation}
In particular, if the coefficients, starting from $f_1$,   in (\ref{fexp}) are real and nonnegative (that is, $f_k\> 0$ for all $k\> 1$), then (\ref{dfdanorm}) leads to
\begin{equation}
\label{dfdanorm1}
    \|\d_\alpha f(\alpha)\|
        \<
    \sum_{k=1}^{+\infty}
    k
    f_k
    \|\alpha\|^{k-1}
    =
    f'(\|\alpha\|),
\end{equation}
where the usual derivative $f'$  of $f$ (as a function of a complex variable) is evaluated at $\|\alpha\|$.

\section{Mixed moments of transfer functions}\label{sec:mixint} For the purposes of Section~\ref{sec:statespace}, the following two lemmas demonstrate the computation of some of the mixed product integrals (\ref{intF15}) in state space. We will only consider strictly proper transfer functions since any proper one is the sum of its limit value at infinity and a strictly proper transfer function, thus allowing the mixed moment of proper transfer functions to be expressed in terms of lower-order moments for strictly proper transfer functions.

\begin{lemma}
\label{lem:FFF}
Suppose $F_k$, with $k = 1,2,3$,  are  strictly proper transfer functions with (in general,  complex) state-space realization matrix triples $(A_k,B_k,C_k)$, where the matrices $A_k$ are Hurwitz.  Then their third-order moment in (\ref{intF123}) is computed as
\begin{equation}
\label{intF123state}
    \bM(F_1, F_2, F_3)
    =
    B_1^* Q_{12}P_{23} C_3^*,
\end{equation}
where $Q_{12}$, $P_{23}$ are the observability and controllability ``cross-Gramians'' of the pairs $(A_1,C_1)$, $(A_2,C_2)$  and  $(A_2,B_2)$, $(A_3,B_3)$, respectively,  satisfying the algebraic Syl\-ves\-ter equations (ASEs)
\begin{equation}
\label{ASEs}
    A_1^* Q_{12} + Q_{12} A_2 + C_1^* C_2 = 0,
    \qquad
    A_2 P_{23} + P_{23} A_3^* + B_2 B_3^* = 0.
\end{equation}
\end{lemma}
\begin{proof}
The moment (\ref{intF123}) can be  represented in terms of the impulse response functions
\begin{equation}
\label{fk}
    f_k(t)
    :=
    I_+(t)
    C_k\re^{tA_k} B_k
    =
    \frac{1}{2\pi}
    \int_\mR
    \re^{i\lambda t}
    F_k(i\lambda)
    \rd \lambda,
    \qquad
    t \in \mR,
\end{equation}
with $k=1,2,3$,
as the convolution of the functions $g_1(t):= f_1(-t)^*$, $g_2(t): =f_2(t)$, $g_3(t): = f_3(-t)^*$ evaluated at $0$:
\begin{equation}
\label{MFFF}
    \bM(F_1,F_2,F_3)
    =
    (g_1 *  g_2  * g_3)(0)
\end{equation}
(in (\ref{fk}), we denote by   $I_+$ the indicator function of the set $\mR_+:= [0, +\infty)$).
This convolution is computed as
\begin{align}
\nonumber
    (g_1 *  g_2  * g_3)&(t)
     =
    \int_{\mR^2}
    f_1(u)^*
    f_2(t+u+v)
    f_3(v)^*
    \rd u
    \rd v\\
\nonumber
    = &
    B_1^*
    \Big(
    I_+(t)
    \int_{\mR_+^2}
    \re^{u A_1^*}
    C_{12}
    \re^{(t+u+v) A_2}
    B_{23}
    \re^{v A_3^*}
    \rd u\rd v\\
\nonumber
    & +
    I_-(t)
    \int_{\{(u,v)\in \mR_+^2:\,  u+v\> -t\}}
    \re^{u A_1^*}
    C_{12}
    \re^{(t+u+v) A_2}
    B_{23}
    \re^{v A_3^*}
    \rd u\rd v
    \Big)
    C_3^*\\
\nonumber
    = &
    B_1^*
    \Big(
        I_+(t)
    \underbrace{
    \int_0^{+\infty}
    \re^{u A_1^*}
    C_{12}
    \re^{uA_2}
    \rd u}_{Q_{12}}
    \re^{tA_2}
    \underbrace{
    \int_0^{+\infty}
    \re^{v A_2}
    B_{23}
    \re^{v A_3^*}
    \rd v}_{P_{23}}\\
\nonumber
    & +
    I_-(t)
    \Big(
    \int_{-t}^{+\infty}
    \re^{u A_1^*}
    C_{12}
    \Big(
    \int_0^{+\infty}
    \re^{(t+u+v) A_2}
    B_{23}
    \re^{v A_3^*}
    \rd v
    \Big)
    \rd u\\
\nonumber
    & +
    \int_0^{-t}
    \re^{u A_1^*}
    C_{12}
    \Big(
    \int_{-t-u}^{+\infty}
    \re^{(t+u+v) A_2}
    B_{23}
    \re^{v A_3^*}
    \rd v
    \Big)
    \rd u
    \Big)
    \Big)
    C_3^*\\
\label{g123}
    = &
    B_1^*
    \Big(
        I_+(t)
        Q_{12}
    \re^{tA_2}
    P_{23}
 +
    I_-(t)
    \Big(
    \re^{-t A_1^*}
    Q_{12}
    P_{23}
    +
    \int_0^{-t}
    \re^{u A_1^*}
    C_{12}
    P_{23}
    \re^{(-t-u) A_3^*}
    \rd u
    \Big)
    \Big)
    C_3^*,
    \qquad
    t \in \mR,
\end{align}
where $I_-:= 1-I_+$ is the indicator function of the set $(-\infty, 0)$, and $C_{12}:= C_1^* C_2$, $B_{23}:= B_2 B_3^*$  for brevity.  The matrices $Q_{12}$, $P_{23}$ in (\ref{g123}) are unique solutions of the corresponding ASEs in (\ref{ASEs}) since the matrices $A_1$, $A_2$, $A_3$ are Hurwitz.
The right-hand side of (\ref{MFFF}) can now be found by letting $t=0$ in (\ref{g123}), which leads to (\ref{intF123state}).
\end{proof}

Note that the matrix $Q_{12}P_{23}$ satisfies the ASE $
    A_1^* Q_{12}P_{23} - Q_{12}P_{23} A_3^* -Q_{12}B_2B_3^* + C_1^*C_2 P_{23} = 0
$,
 obtained by right multiplying the first ASE in (\ref{ASEs}) by $P_{23}$ and left multiplying the second ASE in (\ref{ASEs}) by $Q_{12}$.

\begin{lemma}
\label{lem:F15}
Suppose $F_k$, with $k = 1,\ldots, 5$,  are  strictly proper transfer functions with state-space realization triples $(A_k,B_k,C_k)$ and Hurwitz matrices $A_k$.  Then their moment (\ref{intF15}) is computed as
\begin{equation}
\label{intF15state}
    \bM(F_1,\ldots, F_5)
    =
    B_1^*
    (
        Q_{12}
        P_{25}
    +
    Q_{14}P_{45}
    )
    C_5^*
\end{equation}
in terms of the matrices $Q_{14}$, $P_{25}$ which are found as the unique solutions of the ASEs
\begin{equation}
\label{ASE14_ASE25}
    A_1^* Q_{14}
    +
    Q_{14}
    A_4
    +
        Q_{12}
    P_{23}
    C_3^*C_4
    +
    C_1^*C_2
    P_{23}
    Q_{34}
     =
    0,
    \qquad
    A_2 P_{25}
    +
    P_{25}
    A_5^*
    +
        P_{23}
        C_3^*C_4
        P_{45}
         =0.
\end{equation}
Here, $Q_{12}$, $P_{23}$ are the Gramians from (\ref{ASEs}), and $Q_{34}$, $P_{45}$ are the observability and controllability cross-Gramians for the pairs $(A_3,C_3)$, $(A_4,C_4)$ and  $(A_4,B_4)$, $(A_5,B_5)$, respectively:
\begin{equation}
\label{ASE3445}
    A_3^*
    Q_{34}
    +
    Q_{34}
    A_4
    +
    C_3^*C_4
    =0,
    \qquad
    A_4 P_{45} + P_{45} A_5^* + B_4 B_5^* = 0.
\end{equation}

\end{lemma}
\begin{proof}
By using the impulse response functions (\ref{fk}) with $k=1, \ldots, 5$, the moment (\ref{intF15})  is represented as the convolution of the functions $g_1(t):= f_1(-t)^*$, $g_2(t):= f_2(t)$, $\ldots$, $g_5(t):= f_1(-t)^*$,  evaluated at $0$:
\begin{equation}
\label{MFFFFF}
    \bM(F_1,\ldots,F_5)
    =
    (g_1*\ldots *g_5)(0)
    =
    \int_{\mR}
    (g_1*g_2*g_3)(t)
    (g_4*g_5)(-t)
    \rd t.
\end{equation}
While the convolution $g_1*g_2*g_3$ is found in (\ref{g123}), the convolution $g_4*g_5$ is computed as
\begin{align}
\nonumber
    (g_4*g_5)(t)
    &=
    \int_{\max(0,t)}^{+\infty}
    f_4(s)f_5(s-t)^*
    \rd s
    =
    C_4
    \int_{\max(0,t)}^{+\infty}
    \re^{sA_4}
    B_{45}
    \re^{(s-t)A_5^*}
    \rd s
    C_5^*\\
\nonumber
    &=
    C_4
    \Big(
    I_-(t)
    \underbrace{
    \int_0^{+\infty}
    \re^{s A_4}
    B_{45}
    \re^{sA_5^*}
    \rd s}_{P_{45}}
    \re^{-tA_5^*}
    +
    I_+(t)
    \int_t^{+\infty}
    \re^{sA_4}
    B_{45}
    \re^{(s-t)A_5^*}
    \rd s
    \Big)
    C_5^*\\
\label{g4g5}
    &=
    C_4
    (
    I_-(t)
    P_{45}
    \re^{-tA_5^*}
    +
    I_+(t)
    \re^{tA_4}
    P_{45}
    )
    C_5^*,
\end{align}
where $B_{45}:= B_4 B_5^*$  for brevity, and $P_{45}$ is the Gramian from (\ref{ASE3445}). Therefore, a combination of (\ref{g123}) with (\ref{g4g5}) leads to
\begin{align}
\nonumber
    \int_{\mR}
    (g_1*g_2*g_3)&(t)
    (g_4*g_5)(-t)
    \rd t
     =
    B_1^*
    \Big(
        Q_{12}
        \underbrace{
        \int_0^{+\infty}
        \re^{tA_2}
        P_{23}
        C_{34}
        P_{45}
        \re^{tA_5^*}
        \rd t}_{P_{25}}\\
\nonumber
    & +
    \int_0^{+\infty}
    \Big(
    \re^{t A_1^*}
    Q_{12}
    P_{23}
    +
    \int_0^t
    \re^{u A_1^*}
    C_{12}
    P_{23}
    \re^{(t-u) A_3^*}
    \rd u
    \Big)
    C_{34}
    \re^{tA_4}
    \rd t
    P_{45}
    \Big)
    C_5^*\\
\nonumber
     = &
    B_1^*
    \Big(
        Q_{12}
        P_{25}
    +
    \Big(
    \underbrace{\int_0^{+\infty}
    \re^{t A_1^*}
    Q_{12}
    P_{23}
    C_{34}
    \re^{tA_4}
    \rd t}_{R} \\
\nonumber
    & +
    \int_0^{+\infty}
    \re^{u A_1^*}
    C_{12}
    P_{23}
    \Big(
    \int_u^{+\infty}
    \re^{(t-u) A_3^*}
    C_{34}
    \re^{tA_4}
    \rd t
    \Big)
    \rd u
    \Big)
    P_{45}
    \Big)
    C_5^*\\
\nonumber
     = &
    B_1^*
    \Big(
        Q_{12}
        P_{25}
        +
    \Big(R
    +
    \int_0^{+\infty}
    \re^{u A_1^*}
    C_{12}
    P_{23}
    \Big(
    \underbrace{
    \int_0^{+\infty}
    \re^{v A_3^*}
    C_{34}
    \re^{vA_4}
    \rd v
    }_{Q_{34}}
    \Big)
    \re^{uA_4}
    \rd u
    \Big)
    P_{45}
    \Big)
    C_5^*\\
\label{ggggg}
     = &
    B_1^*
    \Big(
        Q_{12}
        P_{25}
    +
    \Big(
    R
    +
    \underbrace{
    \int_0^{+\infty}
    \re^{u A_1^*}
    C_{12}
    P_{23}
    Q_{34}
    \re^{uA_4}
    \rd u}_{T}
    \Big)
    P_{45}
    \Big)
    C_5^*,
\end{align}
where the matrices
$
    Q_{14}
    :=
    R + T =
    \int_0^{+\infty}
    \re^{u A_1^*}
    (
        Q_{12}
    P_{23}
    C_{34}
    +
    C_{12}
    P_{23}
    Q_{34}
    )
    \re^{uA_4}
    \rd u
$ and $P_{25}$ satisfy the ASEs (\ref{ASE14_ASE25}), and use is made of $C_{34}:=C_3^* C_4$  along with  the Gramian  $Q_{34}$ from (\ref{ASE3445}). The relation (\ref{intF15state}) is now obtained by substituting (\ref{ggggg}) into (\ref{MFFFFF}).
\end{proof}

Note that the third and fifth-order moments of strictly proper stable transfer functions, computed in Lemmas~\ref{lem:FFF}, \ref{lem:F15},  are used in (\ref{chitheta11}).

\end{document}